\documentclass[a4paper,english,fontsize=10pt,parskip=half,abstracton]{scrartcl}
\usepackage{babel}
\usepackage[utf8]{inputenc}
\usepackage[T1]{fontenc}
\usepackage[a4paper,left=20mm,right=20mm,top=30mm,bottom=30mm]{geometry}
\usepackage{amsmath}
\usepackage{amsthm}
\usepackage{amssymb}
\usepackage{aliascnt}
\usepackage{enumerate}
\usepackage{graphicx}
\usepackage[bookmarks=false,
            pdftitle={Refinements of the orthogonality relations for blocks},
            pdfauthor={Benjamin Sambale},
            pdfkeywords={},
            pdfstartview={FitH}]{hyperref}

\newtheorem{Theorem}{Theorem}[section]
\newaliascnt{Lemma}{Theorem}
\newtheorem{Lemma}[Lemma]{Lemma}
\aliascntresetthe{Lemma}
\newaliascnt{Proposition}{Theorem}
\newtheorem{Proposition}[Proposition]{Proposition}
\aliascntresetthe{Proposition}
\newaliascnt{Corollary}{Theorem}

\aliascntresetthe{Corollary}
\newtheorem*{Conjecture}{Conjecture}
\newaliascnt{Definition}{Theorem}
\theoremstyle{definition}

\aliascntresetthe{Definition}
\newaliascnt{Example}{Theorem}
\newtheorem{Example}[Example]{Example}
\aliascntresetthe{Example}

\numberwithin{equation}{section}

\allowdisplaybreaks[1]

\renewcommand{\phi}{\varphi}
\newcommand{\C}{\operatorname{C}}
\newcommand{\N}{\operatorname{N}}
\newcommand{\Z}{\operatorname{Z}}
\newcommand{\cohom}{\operatorname{H}}
\newcommand{\Aut}{\operatorname{Aut}}

\newcommand{\pcore}{\operatorname{O}}
\newcommand{\GL}{\operatorname{GL}}

\newcommand{\Irr}{\operatorname{Irr}}
\newcommand{\IBr}{\operatorname{IBr}}

\newcommand{\tr}{\operatorname{tr}}

\mathchardef\ordinarycolon\mathcode`\:  
 \mathcode`\:=\string"8000
 \begingroup \catcode`\:=\active
   \gdef:{\mathrel{\mathop\ordinarycolon}}
 \endgroup

\newcommand\overmat[3]{%
  \makebox[0pt][l]{$\smash{\overbrace{\phantom{%
    \begin{matrix}#3\end{matrix}}}^{\textstyle{#1}}}$}#2} 

\newcommand{\diagdots}[3][-25]{%
  \rotatebox{#1}{\makebox[0pt]{\makebox[#2]{\xleaders\hbox{$\cdot$\hskip#3}\hfill\kern0pt}}}%
}

\title{Refinements of the orthogonality\\ relations for blocks}
\author{Benjamin Sambale\footnote{Fachbereich Mathematik, TU Kaiserslautern, 67653 Kaiserslautern, Germany, 
\href{mailto:sambale@mathematik.uni-kl.de}{sambale@mathematik.uni-kl.de}}}
\date{\today}

\begin{document}
\frenchspacing
\maketitle

\begin{abstract}\noindent
For a block $B$ of a finite group $G$ there are well-known orthogonality relations for the generalized decomposition numbers. We refine these relations by expressing the generalized decomposition numbers with respect to an integral basis of a certain cyclotomic field. After that, we use the refinements in order to give upper bounds for the number of irreducible characters (of height $0$) in $B$. In this way we generalize results from [Héthelyi-Külshammer-Sambale, 2014]. These ideas are applied to blocks with abelian defect groups of rank $2$. Finally, we address a recent conjecture by Navarro. 
\end{abstract}


\section{Introduction}
Let $B$ be a $p$-block of a finite group $G$ with defect group $D$. Let $k(B):=\lvert\Irr(B)\rvert$ and $l(B):=\lvert\IBr(B)\rvert$. 
Recall that the \emph{height} $i\ge 0$ of $\chi\in\Irr(B)$ is defined by $\chi(1)_p=p^i|G:D|_p$. We denote the number of irreducible characters of height $i$ in $B$ by $k_i(B)$.
For $u\in D$ let $b_u$ be a Brauer correspondent of $B$ in $\C_G(u)$, i.\,e. $(u,b_u)$ is a \emph{$B$-subsection}. Let $Q=(d^u_{\chi\phi}:\chi\in\Irr(B),\phi\in\IBr(b_u))$ be the matrix of generalized decomposition numbers with respect to $(u,b_u)$. By a result of Broué~\cite[Corollary~2]{BroueSanta}, it is known that the rows of $Q$ corresponding to irreducible characters of height $0$ do not vanish (see also \cite[Proposition~1.36]{habil}). Moreover, by the orthogonality relations, $C_u:=Q^{\text{T}}\overline{Q}$ is the Cartan matrix of $b_u$ ($Q^\text{T}$ is the transpose and $\overline{Q}$ is the complex conjugate of $Q$). Since $Q$ consists of algebraic integers in a cyclotomic field, it is possible to bound $k_0(B)$ in terms of $C_u$. If $u\in\Z(D)$, then the subsection is \emph{major} which means that $B$ and $b_u$ have the same defect. Then, by a result of Brauer~\cite[(5H)]{BrauerBlSec2}, every row of $Q$ does not vanish and we get a bound on $k(B)$ in terms of $C_u$. In many cases these bounds do not rely on the \emph{basic set} chosen for $b_u$. Recall that a basic set for $b_u$ is a basis of the $\mathbb{Z}$-module of generalized Brauer characters spanned by $\IBr(b_u)$. If one changes the basic set for $b_u$, then the matrix $C_u$ is transformed into $SC_uS^\text{T}$ where $S\in\GL(l(b_u),\mathbb{Z})$.

We will recall and rephrase two results in that direction which were first published in \cite{HKS}.

\begin{Proposition}\label{outer}
Let $(u,b_u)$ be a $B$-subsection such that $b_u$ dominates a block $\overline{b_u}$ of $\C_G(u)/\langle u\rangle$ with Cartan matrix $\overline{C_u}=(c_{ij})$ up to basic sets.
Then for every positive definite, integral quadratic form
$q(x_1,\ldots,x_{l(b_u)})=\sum_{1\le i\le j\le l(b_u)}{q_{ij}x_ix_j}$ ($q_{ij}\in\mathbb{Z}$)
we have
\[k_0(B)\le|\langle u\rangle|\sum_{1\le i\le j\le l(b_u)}{q_{ij}c_{ij}}.\]
If $(u,b_u)$ is major, then we can replace $k_0(B)$ by $k(B)$ in this formula.
\end{Proposition}
\begin{proof}
It is well-known that the Cartan matrix of $b_u$ is given by $|\langle u\rangle|\overline{C_u}$ (see for example \cite[Theorem~9.10]{Navarro}). Hence, the result follows directly from \cite[Theorem~4.2]{habil}.
\end{proof}

Recall that the Brauer characters of $b_u$ are just inflations of the Brauer characters of $\overline{b_u}$. Therefore, we will often identify $\IBr(\overline{b_u})$ by $\IBr(b_u)$.

If $l(b_u)=1$, then the inequality in \autoref{outer} becomes $k_0(B)\le|R|$ where $R$ is a defect group of $b_u$ (after conjugation one may assume that $R=\C_D(u)$). In this case it is possible to improve the factor $|\langle u\rangle|$ in \autoref{outer}. 
In order to do so, one expresses the generalized decomposition numbers $d^u_{\chi\phi}$ with respect to an integral basis of the cyclotomic field $\mathbb{Q}(e^{2\pi i/|\langle u\rangle|})$. This can be thought of as a discrete Fourier transformation.

\begin{Proposition}\label{mainfusion}
Let $(u,b_u)$ be a $B$-subsection such that $l(b_u)=1$ and $\mathcal{N}:=\N_G(\langle u\rangle,b_u)/\C_G(u)$. Suppose that $b_u$ dominates a block $\overline{b_u}$ of $\C_G(u)/\langle u\rangle$ with defect $d$.
Then
\[k_0(B)\le k_0\bigl(\langle u\rangle\rtimes\mathcal{N}\bigr)p^d.\]
If $(u,b_u)$ is major, then $\lvert\mathcal{N}\rvert\mid p-1$ and
\[k(B)\le\sum_{i=0}^{\infty}{p^{2i}k_i(B)}\le k\bigl(\langle u\rangle\rtimes\mathcal{N}\bigr)|D/\langle u\rangle|.\]
\end{Proposition}

The proof of this result will be given in the next section.
The reason that the group $\langle u\rangle\rtimes\mathcal{N}$ appears in \autoref{mainfusion} is related to the fact that the principal block of this group has a subsection where the generalized decomposition numbers satisfy exactly the same orthogonality relations as in the group $G$ (see \autoref{semidirect} below).

The two theorems above are duals of one another in the sense that the first uses $\C_G(u)/\langle u\rangle$ while the second uses $\langle u\rangle\rtimes\N_G(\langle u\rangle,b_u)/\C_G(u)$. 
The aim of the present paper is to unify both results by making use of the entire group $\N_G(\langle u\rangle,b_u)$ (and dropping the assumption $l(b_u)=1$ in \autoref{mainfusion}). In order to achieve this, we develop two new ideas. The first is to allow also non-integral quadratic forms in \autoref{outer} (see \autoref{lemtensor}) and secondly we choose a suitable integral basis for certain intermediate fields of cyclotomic fields (see \autoref{intbas}). This will simplify the complicated calculations in \cite[Section~5.2]{habil}.
After this has been done, we will apply our main theorem to blocks with abelian defect groups of rank $2$ (see \autoref{2rank}). In Section~4 we give another application by addressing a recent conjecture by Navarro (see \autoref{navarro}). Finally, in the last section some concluding remarks and open questions are presented.

Our notation is mostly standard and can be found in \cite{Navarro} and \cite{habil}. 
In particular, the \emph{inertial quotient} of $B$ is defined by $I(B):=\N_G(D,b_D)/\C_G(D)D$ where $b_D$ is a Brauer correspondent of $B$ in $\C_G(D)$.
In the situation above, we will often regard $\mathcal{N}:=\N_G(\langle u\rangle,b_u)/\C_G(u)$ as a subgroup of $\Aut(\langle u\rangle)$. We will even identify $\mathcal{N}$ with a subgroup of $(\mathbb{Z}/p^n\mathbb{Z})^\times$ where $p^n$ is the order of $u$. Since the Galois group $\mathcal{G}$ of the cyclotomic field of degree $p^n$ is naturally isomorphic to $(\mathbb{Z}/p^n\mathbb{Z})^\times$ as well, we will sometimes consider $\mathcal{N}$ as a subgroup of $\mathcal{G}$. Finally, we identify the elements of $(\mathbb{Z}/p^n\mathbb{Z})^\times$ with integers where we silently compute modulo $p^n$.

In accordance with the definition of $k(B)$ we also use $k(G):=\lvert\Irr(G)\rvert$ and $k_i(G)$ in the obvious meaning.
We write actions from the left; for instance $(^g\chi)(h):=\chi(g^{-1}hg)$ where $g,h\in G$ and $\chi$ is a character of a subgroup of $G$. For a finite group $H$ we often use the abbreviation $H^n:=H\times\ldots\times H$ ($m$ copies).
A cyclic group of order $n$ is denoted by $Z_n$. Moreover, $\mathfrak{S}_n$ is the symmetric group and $\mathfrak{A}_n$ is the alternating group of degree $n$.

\section{Refinements of the orthogonality relations}

\begin{proof}[Proof of \autoref{mainfusion}.]
We first deal with the case $p=2$ which is entirely different. Let $\mathcal{F}$ be the fusion system of $B$. After conjugation, we may assume that $\langle u\rangle$ is a fully $\mathcal{F}$-normalized subgroup of $D$. Then $\C_D(u)$ is a defect group of $b_u$ and $\mathcal{N}\cong\N_D(\langle u\rangle)/\C_D(u)$ (see \cite[Section~5.1]{habil}). Suppose that there exists $\gamma\in\mathcal{N}$ such that $\gamma\equiv -5^n\pmod{|\langle u\rangle|}$ for some $n\in\mathbb{Z}$. Then \cite[Theorem~5.2]{habil} implies $k_0(B)\le 2\lvert\N_D(\langle u\rangle)/\langle u\rangle\rvert=2\lvert\mathcal{N}\rvert p^d$. By taking powers of $\gamma$ we obtain that $u$ is $\mathcal{N}$-conjugate to $u^{-1}$. In particular, the commutator subgroup of $\langle u\rangle\rtimes\mathcal{N}$ is generated by $\langle u^2\rangle$. Hence, $k_0(\langle u\rangle\rtimes\mathcal{N})=2\lvert\mathcal{N}\rvert$ and the claim follows. Next we suppose that there is no element $\gamma$ as above. Since
\[(\mathbb{Z}/|\langle u\rangle|\mathbb{Z})^\times\cong\langle -1+|\langle u\rangle|\mathbb{Z}\rangle\times\langle 5+|\langle u\rangle|\mathbb{Z}\rangle,\]
we see that $\mathcal{N}$ is generated by some $\gamma\equiv 5^{2^n}\pmod{|\langle u\rangle|}$ where $2^{n+2}|\mathcal{N}|=|\langle u\rangle|$. The commutator subgroup of $\langle u\rangle\rtimes\mathcal{N}$ is generated by $u^{5^{2^n}-1}$. Since $5^{2^n}-1$ is divisible by $2^{n+2}$ but not by $2^{n+3}$, we derive $|\langle u^{5^{2^n}-1}\rangle|=|\mathcal{N}|$ and $k_0(\langle u\rangle\rtimes\mathcal{N})=|\langle u\rangle|$. Now the claim already follows from \autoref{outer}.

For the remainder of the proof we assume that $p>2$.
With the notation of \cite[Theorem~5.12]{habil} we have $|\mathcal{N}|=p^sr$ where $r\mid p-1$ and $s\ge 0$. Thus, for the first claim is suffices to show that 
\[k_0\bigl(\langle u\rangle\rtimes\mathcal{N}\bigr)=\frac{|\langle u\rangle|+p^s(r^2-1)}{r}.\]
The inflations from $\mathcal{N}$ yield $p^sr$ linear characters in $\langle u\rangle\rtimes\mathcal{N}$, since $\mathcal{N}$ is cyclic (as a subgroup of $\Aut(\langle u\rangle)$). Now let $1\ne\lambda\in\Irr(\langle u\rangle)$. If the length of the orbit of $\lambda$ under $\mathcal{N}$ is divisible by $p$, then the irreducible characters of $\langle u\rangle\rtimes\mathcal{N}$ lying over $\lambda$ all have positive height. Hence, we may assume that $\lambda^{|\langle u^{p^s}\rangle|}=1$. 
Then, by Clifford theory, $\lambda$ extends in $p^s$ many ways to $\langle u\rangle\rtimes\mathcal{N}_p$ where $\mathcal{N}_p$ is the Sylow $p$-subgroup of $\mathcal{N}$. All these extensions induce to irreducible characters of $\langle u\rangle\rtimes\mathcal{N}$ of height $0$. We have $(|\langle u^{p^s}\rangle|-1)/r$ choices for $\lambda$. Thus, in total we obtain 
\[k_0\bigl(\langle u\rangle\rtimes\mathcal{N}\bigr)=p^sr+p^s\frac{|\langle u^{p^s}\rangle|-1}{r}=\frac{|\langle u\rangle|+p^s(r^2-1)}{r}.\]

Now suppose that $(u,b_u)$ is major. Then $p^d=|D/\langle u\rangle|$, and it suffices to show that $s=0$. This follows for example from \cite[Proposition~I.2.5]{AKO}.
\end{proof}

Suppose that we have dropped the assumption $l(b_u)=1$ in \autoref{mainfusion}. Then the main obstacle in the proof is that $\mathcal{N}:=\N_G(\langle u\rangle,b_u)/\C_G(u)$ can act non-trivially on $\IBr(b_u)$. This means that $\mathcal{N}$ also acts on the columns of the generalized decomposition matrix corresponding to $(u,b_u)$. Consequently, we have to be careful when changing basic sets, since this might disturb the action of $\mathcal{N}$. 
For these reasons the most general result is quite unhandy:

\begin{Theorem}\label{general}
Let $B$ be a $p$-block of a finite group $G$.
Let $(u,b_u)$ be a $B$-subsection such that $b_u$ dominates a block $\overline{b_u}$ of $\C_G(u)/\langle u\rangle$.
Let $l:=l(b_u)$, and let $\overline{C_u}=(c_{\sigma\tau})$ be the Cartan matrix of $\overline{b_u}$ \textup{(}$\sigma,\tau\in\IBr(b_u)$\textup{)}. Let $|\langle u\rangle|=p^n$, $m:=p^{n-1}(p-1)$ and 
$\mathcal{N}:=\N_G(\langle u\rangle,b_u)/\C_G(u)$. 
For $i\in\mathbb{Z}$ let $i'\in\{1,\ldots,p^{n-1}\}$ such that $-i\equiv i'\pmod{p^{n-1}}$.
For $i,j\in\mathbb{Z}$, $\gamma\in\mathcal{N}$ and $\tau\in\IBr(b_u)$ let
\begin{align*}
w_{ij}^\tau(\gamma)&:=|\{\delta\in\mathcal{N}_\tau:p^n\mid i-j\gamma\delta\}|-|\{\delta\in\mathcal{N}_\tau:p^n\mid i+j'\gamma\delta\}|\\
&\mathrel{\phantom{:=}}+|\{\delta\in\mathcal{N}_\tau:p^n\mid i'-j'\gamma\delta\}|-|\{\delta\in\mathcal{N}_\tau:p^n\mid i'+j\gamma\delta\}|
\end{align*}
where $\mathcal{N}_\tau$ is the stabilizer of $\tau$ in $\mathcal{N}$. For $\sigma,\tau\in\IBr(b_u)$ we define $A_{\sigma\tau}=(a_{ij})_{i,j=1}^m$ by
\[a_{ij}:=\sum_{\gamma\in\mathcal{N}/\mathcal{N}_\tau}{c_{\sigma,{^\gamma\tau}}w_{i-1,j-1}^\tau(\gamma)}.\]
Let $M\in\mathbb{Z}^{ml\times ml}$ be the block matrix with blocks $\{A_{\sigma\tau}:\sigma,\tau\in\IBr(b_u)\}$. Let $k\in\mathbb{N}$ be maximal with the property that there exists a matrix $Q\in\mathbb{Z}^{k\times ml}$ without vanishing rows such that $Q^\textup{T}Q=M$. Then $k_0(B)\le k$ and if $(u,b_u)$ is major, we also have $k(B)\le k$.
\end{Theorem}
\begin{proof}
For a fixed $\sigma\in\IBr(b_u)$ we set $d_\sigma^u:=(d_{\chi\sigma}^u:\chi\in\Irr(B))$. Let $\zeta:=e^{2\pi i/p^n}$. Then $d_{\chi\sigma}^u\in\mathbb{Z}[\zeta]$ and $v:=(\zeta^i:i=0,\ldots,m-1)$ is a $\mathbb{Z}$-basis of $\mathbb{Z}[\zeta]$ (see \cite[Proposition~I.10.2]{NeukirchE}). Hence there exists a matrix $A_\sigma\in\mathbb{Z}^{k(B)\times m}$ such that $d_\sigma^u=A_\sigma v$. Let $A:=(A_\sigma:\sigma\in\IBr(b_u))\in\mathbb{Z}^{k(B)\times ml}$. 
By \cite[Proposition~1.36]{habil}, the rows of $A$ corresponding to irreducible characters of height $0$ do not vanish. If $(u,b_u)$ is major, then all the rows of $A$ do not vanish. Consequently, it suffices to show that $A^\text{T}A=M$. This matrix is built from blocks of the form $A_\sigma^\text{T}A_\tau$. Thus we need to show that $A_\sigma^\text{T}A_\tau=A_{\sigma\tau}$.

The orthogonality relation implies 
\[v^\text{T}A_\sigma^\text{T}A_\tau\overline{v}=(d_\sigma^u)^\text{T}\overline{d^u_\tau}=p^nc_{\phi\tau},\]
since $p^n\overline{C_u}$ is the Cartan matrix of $b_u$.
This can be rewritten in the form
\[(v\otimes \overline{v})X=p^nc_{\sigma\tau}\]
where $\otimes$ denotes the Kronecker product and $X$ is the vectorization $A_\sigma^\text{T}A_\tau$. 
If $\gamma,\delta\in(\mathbb{Z}/p^n\mathbb{Z})^\times$ such that $\gamma^{-1}\delta\notin\mathcal{N}$, then $(d^{u^\gamma}_\sigma)^\text{T}d^{u^{\delta}}_\tau=0$. This gives another equation
\[(\gamma(v)\otimes\delta(\overline{v}))X=0.\]
Note that we consider $\gamma$ and $\delta$ here as Galois automorphisms of $\mathbb{Q}(\zeta)$.
Now let $\gamma^{-1}\delta\in\mathcal{N}$. Then $d^{u^\delta}_\tau=d^{u^\gamma}_{\tau'}$ with $\tau':={^{\gamma\delta^{-1}}\tau}\in\IBr(b_u)$. In this case we obtain
\[(\gamma(v)\otimes\delta(\overline{v}))X=p^nc_{\sigma\tau'}.\]
Let $V$ be the matrix with rows $(\gamma(v)\otimes\delta(\overline{v}):\gamma,\delta\in(\mathbb{Z}/p^n\mathbb{Z})^\times)$. Then we have a linear system $VX=b$ where $b$ contains zeros and some $p^nc_{\sigma\tau'}$ as above. Observe that $W=(\gamma(v):\gamma\in
(\mathbb{Z}/p^n\mathbb{Z})^\times)$ is a Vandermonde matrix and $V=W\otimes \overline{W}$. 
By \cite[Lemma~5.5]{habil}, we have $W^{-1}=p^{-n}(\gamma(t_{i-1}))$ with $t_i:=\zeta^{-i}-\zeta^{i'}$ where $\gamma$ now runs through the columns of $W^{-1}$. Since $X=V^{-1}b=(W^{-1}\otimes\overline{W}^{-1})b$, the entry of $A_\sigma^\text{T}A_\tau$ at position $(i+1,j+1)$ is
\begin{align}\label{cycsum}
p^{-n}\sum_{\substack{\gamma,\delta\in(\mathbb{Z}/p^n\mathbb{Z})^\times,\\\gamma^{-1}\delta\in\mathcal{N}}}{\gamma(t_j)\delta(\overline{t_i})c_{\sigma,{^{\gamma\delta^{-1}}\tau}}}&=p^{-n}\sum_{\gamma\in\mathcal{N}}\sum_{\mu\in(\mathbb{Z}/p^n\mathbb{Z})^\times}{
\mu(\overline{t_i}\gamma(t_j))c_{\sigma,{^{\gamma}\tau}}}\nonumber\\
&=p^{-n}\sum_{\gamma\in\mathcal{N}/\mathcal{N}_\tau}c_{\sigma,{^\gamma\tau}}\sum_{\delta\in\mathcal{N}_\tau}\sum_{\mu\in(\mathbb{Z}/p^n\mathbb{Z})^\times}{\mu(t_i(\gamma\delta)(\overline{t_j}))}.
\end{align}
Now, 
\[t_i(\gamma\delta)(\overline{t_j})=(\zeta^{-i}-\zeta^{i'})(\zeta^{\gamma\delta j}-\zeta^{-\gamma\delta j'})=\zeta^{-i+\gamma\delta j}-\zeta^{-i-\gamma\delta j'}+\zeta^{i'-\gamma\delta j'}-\zeta^{i'+\gamma\delta j}.\]
Since the $p^n$-th cyclotomic polynomial is given by $\Phi_{p^n}(x)=x^{(p-1)p^{n-1}}+x^{(p-2)p^{n-1}}+\ldots+x^{p^{n-1}}+1$, it follows that
\[\sum_{\mu\in(\mathbb{Z}/p^n\mathbb{Z})^\times}{\mu(\zeta^{-i+\gamma\delta j})}=\begin{cases} m&\text{if }p^n\mid {i-\gamma\delta j},\\
0&\text{if }p^{n-1}\nmid {i-\gamma\delta j},\\
-p^{n-1}&\text{otherwise}.\end{cases}\]
An application of this leads us to
\[\sum_{\delta\in\mathcal{N}_\tau}\sum_{\mu\in(\mathbb{Z}/p^n\mathbb{Z})^\times}{\mu(\zeta^{-i+\gamma\delta j})}=p^n|\{\delta\in\mathcal{N}_\tau:p^n\mid i-\gamma\delta j\}|-p^{n-1}|\{\delta\in\mathcal{N}_\tau:p^{n-1}\mid i-\gamma\delta j\}|.\]
We get similar expressions for the other terms $i+\gamma\delta j'$, $i'-\gamma\delta j'$ and $i'+\gamma\delta j$ in  \eqref{cycsum}. Since $i+i'\equiv 0\equiv j+j'\pmod{p^{n-1}}$, we have $i-\gamma\delta j\equiv i+\gamma\delta j'\equiv -i'+\gamma\delta j'\equiv -i'-\gamma\delta j\pmod{p^{n-1}}$. Hence, the terms of the form $p^{n-1}|\{\ldots\}|$ in \eqref{cycsum} cancel out each other. It turns out that \eqref{cycsum} coincides with $a_{ij}$. This proves the claim.
\end{proof}

The important point of \autoref{general} is the fact that the matrix $M$ is uniquely determined by $\overline{C_u}$, $\mathcal{N}$ and the action on $\IBr(b_u)$. Therefore, the result can indeed be seen as a refinement of the usual orthogonality relations for blocks. The numbers $w_{ij}^\tau(\gamma)$ can be evaluated individually (see \cite[Lemma~5.9]{habil}), but the entire matrix $M$ seems to have an extremely opaque shape.
Nevertheless, from an algorithmic point of view, it is trivial to construct $M$. However, the computation of $k$ is a hard problem. There is a sophisticated algorithm by Plesken~\cite{Plesken} which can be used if $ml$ is small. In general one can use estimates of $k$ as in \autoref{outer}. Note also that $M$ usually does not have full rank. On the one hand, this is because the generalized decomposition numbers usual lie in proper subfields of $\mathbb{Q}(\zeta)$. On the other hand, $A_{\sigma\tau}=A_{^\gamma\sigma,{^\gamma\tau}}$ for $\sigma,\tau\in\IBr(b_u)$ and $\gamma\in\mathcal{N}$. Hence, before one attempts to compute $k$, one should replace $M$ by a matrix of smaller size. In particular, one can apply the LLL reduction algorithm for lattices.
We give a complete example.

\begin{Example}\label{examp}
Let $B$ be the principal $3$-block of $G={^2F_4(2)'}$. Then $B$ has extraspecial defect group $D$ of order $p^3$ and exponent $p$. It can be seen from the Atlas~\cite{Atlas} that $G$ has only one conjugacy class of elements of order $3$. In particular, there is only one non-trivial $B$-subsection $(u,b_u)$ which must necessarily be major (in fact, the fusion system of $B$ has been determined in \cite{ExtraspecialExpp}). The group $\C_G(u)/\langle u\rangle$ is isomorphic to $Z_3^2\rtimes Z_4$ and the Cartan matrix $\overline{C_u}$ of $\overline{b_u}$ is given by $\overline{C_u}=(2+\delta_{ij})_{i,j=1}^4$ where $\delta_{ij}$ is the Kronecker delta (in particular, $l(b_u)=4$). Now $b_u$ is covered only by the principal block of $N:=\N_G(\langle u\rangle)$, since this subgroup has only one block. Moreover, $|N/\C_G(u)|=2$ and $l(N)=5$. Clifford theory reveals that $b_u$ has two $N$-invariant irreducible Brauer characters and the other two characters are conjugate under $N$. By the shape of $\overline{C_u}$, we may assume that the last two Brauer characters are stable. Then the matrix $M$ from \autoref{general} is given as
\[\begin{pmatrix}
8&1&7&-1&6&.&6&.\\
1&2&-1&-2&.&.&.&.\\
7&-1&8&1&6&.&6&.\\
-1&-2&1&2&.&.&.&.\\
6&.&6&.&9&.&6&.\\
.&.&.&.&.&.&.&.\\
6&.&6&.&6&.&9&.\\
.&.&.&.&.&.&.&.
\end{pmatrix}\]
and this reduces to
\[\begin{pmatrix}
2&1&1&1\\
1&5&2&2\\
1&2&5&2\\
1&2&2&8
\end{pmatrix}.\]
Now Plesken's algorithm gives $k=15$. Therefore $k(B)\le 15$. In comparison, \autoref{outer} implies only $k(B)\le 18$ and the actual value is $k(B)=13$ (\autoref{mainfusion} is not applicable). The gap between $13$ and $15$ can be explained as follows: $B$ contains four irreducible characters of height $1$. By Brauer's theory, the so-called \emph{contributions} of these characters have positive $p$-adic valuation (see \cite[Proposition~1.36]{habil}). This implies that the generalized decomposition numbers corresponding to these characters cannot be too “small”. However, note that the expression $\sum_{i=0}^{\infty}{p^{2i}k_i(B)}\le 15$ from \autoref{mainfusion} would not be true here. But we can still say something in this direction: Plesken's algorithm also shows that there is essentially only one solution $Q^\text{T}Q=M$ with $Q\in\mathbb{Z}^{15\times 4}$. For this solution matrix $Q$ we can revert all the calculations to obtain the generalized decomposition numbers. After that we can compute the contributions. This shows that $k(B)=15$ can only happen if $k(B)=k_0(B)$. Similarly, $k(B)=14$ would imply $k_0(B)=12$. 
\end{Example}

Our next goal is to simplify the situation if the Brauer characters are stable. This will generalize \autoref{mainfusion}. Therefore, we need to revisit and simplify the calculations in \cite[Section~5.2]{habil}.

\begin{Lemma}\label{lemtensor}
Let $q$ be a \textup{(}not necessarily integral\textup{)} quadratic form such that $q(x)\ge 1$ for all non-zero integral vectors $x$. Let $q_A(x)=\sum{x_i^2}-\sum{x_ix_{i+1}}$ be the quadratic form corresponding to a Dynkin diagram of type $A$. Then we have $(q\otimes q_A)(x)\ge 1$ for every integral $x\ne 0$.
\end{Lemma}
\begin{proof}
Since $q_A$ is positive definite, there is no doubt that $q\otimes q_A$ is positive definite. 
However, we may have $0<(q\otimes q_A)(x)<1$.
It is known that $q\otimes q_A$ and $q_A\otimes q$ are isomorphic (the Gram matrices are permutation similar). Thus it suffices to show $(q_A\otimes q)(x)\ge 1$ for all integral $x\ne 0$. If $q$ has rank $n$ and $q_A$ has rank $m$, then we can write $x=(x_1,\ldots,x_m)$ with $x_i\in\mathbb{Z}^{n}$. It follows that
\[(q_A\otimes q)(x)=\frac{1}{2}\sum_{i=1}^{m-1}{q(x_i-x_{i+1})}+\frac{1}{2}\bigl(q(x_1)+q(x_m)\bigr).\]
The claim follows easily from the hypothesis on $q$.
\end{proof}

\autoref{lemtensor} is relevant, because in general a tensor product of two integral quadratic forms is not integral anymore. For, if $q=\sum_{i\le j}{q_{ij}x_ix_j}$ is integral, the coefficients $q_{ij}$ with $i\ne j$ lie in $\frac{1}{2}\mathbb{Z}$. Hence, the coefficients of $q_1\otimes q_2$ only lie in $\frac{1}{4}\mathbb{Z}$. Moreover, the minimum of a tensor product of quadratic forms is not necessarily the product of the minima (but “most” of the time it is, see \cite[Section~7.1]{Kitaoka}).

\begin{Lemma}\label{intbas}
Let $\zeta\in\mathbb{C}$ be a primitive $p^n$-th root of unity for an odd prime power $p^n$, and let $\mathcal{G}$ be the Galois group of $\mathbb{Q}(\zeta)$ over $\mathbb{Q}$. 
Let $\mathcal{N}\le\mathcal{G}$ such that $\lvert\mathcal{N}\rvert=p^sr$ with $r\mid p-1$ and $s\ge 0$.
For $i\ge 1$ let $S_i$ be a set of representatives for the $\pcore_{p'}(\mathcal{N})$-orbits of $\{1\le j\le p^i:(j,p)=1\}$. Let 
\begin{align*}
T_1&:=S_1,\\
T_i&:=pT_{i-1}\cup \{s+jp^{i-1}:s\in S_{i-1},\ j=0,\ldots,p-2\}\qquad\text{for }i\ge 2.
\end{align*}
Then 
\[\widetilde{T}_n:=\Bigl\{\sum_{\gamma\in\pcore_{p'}(\mathcal{N})}{\gamma(\zeta^t)}:t\in p^sT_{n-s}\Bigr\}\] 
is an integral basis for the fixed field of $\mathcal{N}$.
\end{Lemma}
\begin{proof}
First we reduce to the case $s=0$. Thus, assume $s>0$. Since $p>2$, $\mathcal{G}$ is cyclic and $\mathcal{N}$ is uniquely determined by $r$ and $s$. Let $F$ be the fixed field of $\mathcal{N}$. By Galois theory, it follows that $[\mathbb{Q}(\zeta):F]=\lvert\mathcal{N}\rvert$ and $F\subseteq\mathbb{Q}(\zeta^{p^s})$. Since $\pcore_p(\mathcal{N})$ acts trivially on $\mathbb{Q}(\zeta^{p^s})$, $F$ is just the fixed field of $\pcore_{p'}(\mathcal{N})$ in $\mathbb{Q}(\zeta^{p^s})$. We may therefore assume that $s=0$ in the following.

The group $\mathcal{N}$ acts on $\{1\le j\le p^i:(j,p)=1\}$ via $^\gamma j:=\gamma j\pmod{p^i}$. Hence, the sets $S_i$ and $T_i$ are well-defined.
Since $\mathcal{N}$ is a $p'$-group, the canonical map $\mathcal{N}\to(\mathbb{Z}/p^i\mathbb{Z})^\times$ is injective for $i\ge 1$. Therefore, $\mathcal{N}$ acts always semiregularly. 
Let $m:=p^{n-1}(p-1)$. 
First we show that $V:=\{\gamma(\zeta^t):\gamma\in\mathcal{N},t\in T_n\}$ is an integral basis of $\mathbb{Q}(\zeta)$. 
We argue by induction on $n$.
It is well-known that $\zeta,\zeta^2,\ldots,\zeta^m$ is an integral basis of $\mathbb{Q}(\zeta)$.
Hence, there is nothing to do for $n=1$. Now let $n\ge 2$. Then by induction we have $|V|=p^{n-2}(p-1)+p^{n-2}(p-1)^2=m$. Thus, it suffices to show that each $\zeta^i$ 
can be written as an integral linear combination of $V$. If $p\mid i$, then induction shows that $\zeta^i$ can be written in terms of $\gamma(\zeta^t)$ with $t\in pT_{n-1}$. Thus, we may assume that $p\nmid i$. 
Choose $j$ such that $1\le i-jp^{n-1}\le p^{n-1}$. By the definition of $S_{n-1}$, there exist $\gamma\in\mathcal{N}$ and $k\in\mathbb{Z}$ such that $\gamma i-kp^{n-1}\in S_{n-1}$. If $0\le k\le p-2$, then $\gamma i\in T_n$ and the claim follows. Now let $k=p-1$. Then $\gamma i+lp^{n-1}\in T_n$ for all $l\in\{1,\ldots,p-1\}$.
Since $1=-\sum_{l=1}^{p-1}{\zeta^{lp^{n-1}}}$, we have $\zeta^{\gamma i}=-\sum_{l=1}^{p-1}{\zeta^{\gamma i+lp^{n-1}}}$. Therefore, it follows that $V$ is in fact an integral basis for $\mathbb{Q}(\zeta)$. 

The elements $\tr(\zeta^t):=\sum_{\gamma\in\mathcal{N}}{\gamma(\zeta^t)}$ with $t\in T_n$ certainly lie in the ring of integers of the fixed field $F$ of $\mathcal{N}$. Conversely, let $x\in F$ be an algebraic integer. Then by the arguments above, we can write $x=\sum_{v\in V}{\alpha_vv}$ with $\alpha_v\in\mathbb{Z}$. It is now obvious that $x=\sum_{t\in T_n}{\alpha_{\zeta^t}\tr(\zeta^t)}$. Consequently, $\widetilde{T}_n$ is a generating set for the ring of integers of $F$. By Galois theory, $|F:\mathbb{Q}|=p^{n-1}(p-1)/|\mathcal{N}|=|T_n|$. Hence, the elements really form an integral basis and we are done.
\end{proof}

\begin{Lemma}\label{semidirect}
Let $p^n$ be an odd prime power, and let $\gamma\in(\mathbb{Z}/p^n\mathbb{Z})^\times$ of order $p^sr$ with $r\mid p-1$ and $s\ge 0$.
Let $B$ be the principal $p$-block of 
\[G:=\langle u,x\mid u^{p^n}=x^{p^sr}=1,\ xux^{-1}=u^\gamma\rangle,\]
and let $b_u$ be the principal block of $\C_G(u)=\langle u\rangle$. Then $\IBr(b_u)=\{\phi\}$ and the generalized decomposition numbers $(d^u_{\chi\phi}:\chi\in\Irr(B))$ can be written in the form $A\widetilde{T}_n$ where $A$ is an integral matrix and $\widetilde{T}_n$ is the basis from \autoref{intbas}. Let $M=(m_{ij}):=A^\textup{T}A\in\mathbb{Z}^{t\times t}$ with $t=p^{n-s-1}(p-1)/r$. Then for the quadratic form $q(x)=\sum_{1\le i\le j\le t}{q_{ij}x_ix_j}$ corresponding to the Dynkin diagram of type $A_t$ we have
\[\sum_{1\le i\le j\le t}{q_{ij}m_{ij}}=k_0(B).\]
\end{Lemma}
\begin{proof}
Observe that $B$ is the only block of $G$.
It is obvious that $\C_G(u)=\langle u\rangle$ and $l(b_u)=1$. Note that $\phi$ is the trivial Brauer character of $\langle u\rangle$. The irreducible characters of $B$ can be obtained as in the proof of \autoref{mainfusion}. 
If $\chi\in\Irr(G)$ is an inflation from $\langle x\rangle$, then $d^u_{\chi\phi}=1$. Now assume that $\chi$ lies over a non-trivial character $\lambda\in\Irr(\langle u\rangle)$. Then we have $\chi(u)=\lambda_1(u)+\ldots+\lambda_t(u)$ where $\{\lambda_1,\ldots,\lambda_t\}$ is the orbit of $\lambda$ under $\langle x\rangle$. If $p\mid t$, then $\chi$ has positive height and $d^u_{\chi\phi}=0$ by \autoref{intbas}. Otherwise, $d^u_{\chi\phi}$ has the form \[\tr(\zeta_i):=\sum_{j=1}^{r}{\zeta_i^{\gamma^j}}\]
where $\zeta_1,\ldots,\zeta_m$ is a set of representatives of the non-trivial $p^{n-s}$-th roots of unity under the action of $\gamma$ ($m=(p^{n-s}-1)/r$).
Since every $\lambda$ has $p^s$ extensions to $\langle u,x^r\rangle$, every $\zeta_i$ appears $p^s$ times in the generalized decomposition matrix. 

Now we need to express the numbers $d^u_{\chi\phi}$ in terms of $\widetilde{T}_n$ from \autoref{intbas}. With the notation from \autoref{intbas} we have $1=-\sum_{i=1}^{p-1}{\zeta^{p^{n-1}i}}$. Hence the first rows of $A$ are given by $(-1,\ldots,-1,0,\ldots,0)$ with $(p-1)/r$ entries $-1$. If $\zeta_i$ above has order $p$, then $\tr(\zeta_i)$ is just an element of $\widetilde{T}_n$. Now suppose that $\zeta_i$ has order greater than $p$. Then it can happen that $\tr(\zeta_i)$ does not belong to $\widetilde{T}_n$. But in this case we have $\zeta_i=-\sum_{j=1}^{p-1}{\zeta_i\zeta^{jp^{n-1}}}$ and by construction all the elements $\tr(\zeta_i\zeta^{jp^{n-1}})$ lie in $\widetilde{T}_n$. Moreover, for $i'\ne i$, the sets $\{\tr(\zeta_i\zeta^{jp^{n-1}}):j=1,\ldots,p-1\}$ and $\{\tr(\zeta_{i'}\zeta^{jp^{n-1}}):j=1,\ldots,p-1\}$ are disjoint. 
Consequently, we may order these basis elements in such way that the corresponding rows of $A$ have the form
\[\begin{pmatrix}
&1\\
\\
&\multicolumn{10}{c}{\diagdots[-13]{19em}{.4em}}\\
\\
0&&&&&&&&&&1\\[1mm]
&-1&\cdots&-1\\
&&&&-1&\cdots&-1&&&&&0\\
&&&&&&&\ddots\\
&&&&&&&&-1&\cdots&-1
\end{pmatrix}.\]
Each of these rows appear $p^s$ times.
Now we put all the ingredients together and it follows that $A^\text{T}A$ is a block diagonal matrix. The first block is $(p^sr+p^s\delta_{ij})_{i,j=1}^{(p-1)/r}$ and corresponds to the basis elements of order $p$. All other blocks are given by $p^s(1+\delta_{ij})_{i,j=1}^{p-1}$. The number of these blocks is $(p^{n-s-1}-1)/r$. Hence, we have
\begin{align*}
\sum_{1\le i\le j\le t}{q_{ij}m_{ij}}&=\Bigl(\frac{p-1}{r}\bigl(p^sr+p^s\bigr)-\bigl(\frac{p-1}{r}-1\bigr)p^sr\Bigr)+\frac{p^{n-1}-p^s}{r}\bigl(2(p-1)-(p-2)\bigr)\\
&=\frac{p-1}{r}p^s+p^sr+\frac{p^n-p^{s+1}}{r}=\frac{p^n+p^s(r^2-1)}{r}=k_0(B)
\end{align*}
(cf. proof of \autoref{mainfusion}).
\end{proof}

\begin{Theorem}\label{stablecase}
Let $B$ be a $p$-block of a finite group $G$.
Let $(u,b_u)$ be a $B$-subsection such that $b_u$ dominates a block $\overline{b_u}$ of $\C_G(u)/\langle u\rangle$.
Suppose that all irreducible Brauer characters of $b_u$ are stable under $\mathcal{N}:=\N_G(\langle u\rangle,b_u)/\C_G(u)$. Let $l:=l(b_u)$, and let $\overline{C_u}$ be the Cartan matrix of $\overline{b_u}$ up to basic sets. 
Then for every positive definite, integral quadratic form
$q(x_1,\ldots,x_l)=\sum_{1\le i\le j\le l}{q_{ij}x_ix_j}$
we have
\[k_0(B)\le k_0\bigl(\langle u\rangle\rtimes\mathcal{N}\bigr)\sum_{1\le i\le j\le l}{q_{ij}c_{ij}}.\]
If $(u,b_u)$ is major, then
\[k(B)\le k\bigl(\langle u\rangle\rtimes\mathcal{N}\bigr)\sum_{1\le i\le j\le l}{q_{ij}c_{ij}}.\]
\end{Theorem}
\begin{proof}
First of all, note that we only need to know $\overline{C_u}$ up to basic sets, because changing the basic set has the same effect as changing the quadratic from $q$ accordingly (see \cite[p. 83]{KuelshammerWada}). 
Now we discuss the special case $p=2$. Similarly as in the proof of \autoref{mainfusion}, we need to distinguish two cases. In the first case $\mathcal{N}$ contains an element $\gamma\equiv -5^n\pmod{|\langle u\rangle|}$. Then \cite[Theorem~5.2]{habil} implies
\[k_0(B)\le 2|\mathcal{N}|\sum_{1\le i\le j\le l}{q_{ij}c_{ij}}\]
and the claim follows as in \autoref{mainfusion}. In the second case there is no such element $\gamma$. Here, $k_0(\langle u\rangle\rtimes\mathcal{N})=|\langle u\rangle|$ and the claim follows directly from \autoref{outer}.

In the following we will assume $p>2$. Let $|\langle u\rangle|=p^n$ and $m:=p^{n-1}(p-1)$. In the definition of $a_{ij}$ in \autoref{general} we may take $\gamma=1$. It follows that there is a matrix $T=(t_{ij})_{i,j=1}^m$ such that $A_{\sigma\tau}=c_{\sigma\tau}T$ for all $\sigma,\tau\in\IBr(b_u)$. 
This gives $M=(m_{ij})=\overline{C_u}\otimes T$. It is easy to see that $T$ is exactly the matrix we would get from \autoref{general} applied to the group $\langle u\rangle\rtimes\mathcal{N}$. By using the integral basis from \autoref{intbas}, we can replace $T$ by a matrix of size $\frac{m}{|\mathcal{N}|}\times\frac{m}{|\mathcal{N}|}$ (constructed in the proof of \autoref{semidirect}). We still denote this smaller matrix by $T$.
If $q_A(x)=\sum_{1\le i\le j\le \frac{m}{|\mathcal{N}|}}{f_{ij}x_ix_j}$ is the quadratic form corresponding to the Dynkin diagram of type $A_{m/|\mathcal{N}|}$, then \autoref{semidirect} gives
\[\sum_{1\le i\le j\le \frac{m}{|\mathcal{N}|}}{f_{ij}t_{ij}}=k_0\bigl(\langle u\rangle\rtimes\mathcal{N}\bigr).\]

Now let $U^\text{T}U=M$ for a matrix $U=(u_{ij})\in\mathbb{Z}^{k\times \frac{ml}{|\mathcal{N}|}}$ without vanishing rows. We need to bound $k$. 
By \autoref{lemtensor}, $q^*:=q\otimes q_A$ satisfies $q^*(x)\ge 1$ for every $x\in\mathbb{Z}^{\frac{ml}{|\mathcal{N}|}}\setminus\{0\}$. Recall that the Gram matrix of $q^*$ is given by
\[(z_{ij})=\frac{1}{2}\begin{pmatrix}
2q_{11}&q_{12}&\cdots&q_{1l}\\
q_{21}&2q_{22}&\ddots&\vdots\\
\vdots&\ddots&\ddots&q_{l-1,l}\\
q_{l1}&\cdots& q_{l,l-1}&2q_{ll}
\end{pmatrix}\otimes\frac{1}{2}\begin{pmatrix}
2&-1&0&\dots&0\\
-1&\ddots&\ddots&\ddots&\vdots\\
0&\ddots&\ddots&\ddots&0\\
\vdots&\ddots&\ddots&\ddots&-1\\
0&\dots&0&-1&2
\end{pmatrix}.\]
This implies
\[k\le\sum_{i=1}^k{q^*(u_{i1},\ldots,u_{i,\frac{ml}{|\mathcal{N}|}})}=\sum_{i=1}^k\sum_{1\le r,s\le \frac{ml}{|\mathcal{N}|}}{z_{rs}u_{ir}u_{is}}=\sum_{1\le r,s\le \frac{ml}{|\mathcal{N}|}}{z_{rs}m_{rs}}=k_0\bigl(\langle u\rangle\rtimes\mathcal{N}\bigr)\sum_{1\le i\le j\le l}{q_{ij}c_{ij}}.\qedhere\]
\end{proof}

The hypothesis of \autoref{stablecase} is met surprisingly often. This can be explained by the following argument. By \cite[Corollary~9.21]{Navarro}, 
$B_u:=b_u^{\N_G(\langle u\rangle,b_u)}$ is the only block of $\N_G(\langle u\rangle,b_u)$ covering $b_u$. 
Clifford theory gives strong relations between $k(b_u)$ and $l(b_u)$ on the one hand and $k(B_u)$ and $l(B_u)$ on the other hand. In particular, all characters in $\IBr(b_u)$ are $\N_G(\langle u\rangle,b_u)$-stable if and only if $l(B_u)=l(b_u)r$ where $r$ is the $p'$-part of $|\mathcal{N}|$ (see \cite[Theorems~8.11, 8.12 and Corollary~8.20]{Navarro}). In general, the numbers $l(b_u)$, $l(B_u)$ and $|\mathcal{N}|$ do not determine the orbit lengths on $\IBr(b_u)$ uniquely. For example, $l(b_u)=l(B_u)=18$ and $|\mathcal{N}|=24$ allow orbit lengths $3,3,12$, but also $2,8,8$.

In the next section we will discuss a special case where not all Brauer characters are stable.

\section{Abelian defect groups of rank $2$}
Suppose that the defect group $D$ of the block $B$ is metacyclic. If $p=2$, then Alperin's Weight Conjecture holds for $B$ (see \cite[Corollary~8.2]{habil}). Thus, let $p>2$. Then by Watanabe~\cite{WatanabeAWC} (cf. \cite[Theorem~8.8]{habil}), Alperin's Weight Conjecture also holds provided $D$ is non-abelian (in addition). Hence, it is of interest to study the case $D\cong Z_{p^n}\times Z_{p^m}$ more closely. Here even in the smallest case $D\cong Z_3^2$ Alperin's Weight Conjecture is open (cf. Kiyota~\cite{Kiyota} and Watanabe~\cite{WatanabeSD16}). 
Therefore, any new information is valuable. In this section we will prove the following.

\begin{Theorem}\label{2rank}
Let $B$ be a $p$-block of a finite group $G$ with defect group $D\cong Z_{p^n}\times Z_{p^m}$ and inertial quotient $I(B)=I\times J$ such that $D\rtimes I(B)\cong (Z_{p^n}\rtimes I)\times(Z_{p^m}\rtimes J)$. Let $b$ be the Brauer correspondent of $B$ in $\N_G(D)$. Then $k(B)\le k(b)$ and $l(B)\le l(b)\le|I(B)|$.
\end{Theorem}

Before we begin with the proof we make a few remarks. Alperin's Weight Conjecture for blocks with abelian defect groups asserts that $l(B)=l(b)$ (with the notation from \autoref{2rank}). Hence, $l(B)\le l(b)$ is a sharp bound and so is $k(B)\le k(b)$ by standard arguments (see proof of \autoref{2rank}). Our result is trivial if $p=2$, because $B$ is nilpotent in this case. If $p=3$, then we have $I(B)\le Z_2^2$. Hence, results by Usami and Puig~\cite{Usami23I,UsamiZ2Z2} already imply Alperin's Weight Conjecture. For $p\ge 5$, \autoref{2rank} is something new (at least to the author's knowledge). 

More generally, let $B$ be a block with abelian defect group $D$ of rank $2$ such that $I(B)$ acts decomposably on $D$. 
Then $I(B)\le Z_{p-1}^2$. Apart from \autoref{2rank}, it may happen that $I(B)$ acts semiregularly on $D\setminus\{1\}$.
Then $I(B)$ is cyclic and \cite[Theorem~5]{SambaleC3} implies the optimal bound $l(B)\le|I(B)|=l(b)$. 
Unfortunately, there are other “mixed” cases where we cannot say much at the moment. For example, let $D\cong Z_7^2$, and let $I(B)\cong Z_6$ act on the first factor faithfully and as an inversion on the second factor. Then the ideas of the present paper show $k(B)\le 25$, but actually $k(B)=22$. Moreover, in general it seems vastly more difficult to prove the converse inequality $l(B)\ge l(b)$ (this is a main obstacle for $D\cong Z_3^2$).

Now we begin with the proof of \autoref{2rank}. The key is to show that for blocks with cyclic defect groups, the action on the Brauer characters is not arbitrary.
We say that an action is \emph{$\frac{1}{2}$-transitive}, if all its orbits have the same length. Note that we also consider the trivial action as $\frac{1}{2}$-transitive.

\begin{Proposition}\label{halftrans}
Let $G$ be a finite group with $N\unlhd G$ such that $G/N$ is cyclic. Let $B$ be a $G$-stable block of $N$ with cyclic defect group. Then the action of $G$ on $\IBr(B)$ is $\frac{1}{2}$-transitive. 
In particular, the action is trivial if $B$ is the principal block.
\end{Proposition}
\begin{proof}
Let us assume the contrary. By replacing $G$ with a group $H$ such that $N\le H\le G$, we may assume that the action of $G$ on $\IBr(B)$ has fixed points and all non-trivial orbits have the same length $q$. We may assume further that $q$ is a prime and $|G/N|$ is a power of $q$. By Dade's theory of cyclic defect groups~\cite{Dade} (see also \cite[Theorem~8.6]{habil}), $l(B)<p$. In particular, $q\ne p$.
Let $N\le H\le G$ such that $|G/H|=q$. Then $H$ acts trivially on $\IBr(B)$. 
Since $H/N$ is cyclic, it is well-known that all irreducible Brauer characters of $B$ extend to $H$ (see \cite[Theorem~8.12 and Corollary~8.20]{Navarro}). By \cite[Theorem~9.4]{Navarro}, there is a block $B_H$ of $H$ covering $B$ such that $\IBr(B_H)$ contains an extension of every irreducible Brauer character of $B$. In particular, $\IBr(B_H)$ also contains a $G$-invariant character and $B_H$ is $G$-stable. Moreover, $B$ and $B_H$ have the same defect group. Since $\IBr(B_H)$ also contains non-stable characters, we may replace $N$ by $H$ and $B$ by $B_H$. Thus we may assume that $|G/N|=q$ in the following.

Let $D$ be a defect group of $B$. Since $G$ acts on the defect groups of $B$, the Frattini argument implies $G=N\N_G(D)$. Hence, we may regard $G/N$ as a quotient of $\N_G(D)$. Let $b$ be a Brauer correspondent of $B$ in $\N_N(D)$. By work of Dade~\cite{Dadeconj2} (see also \cite[Lemma~4.1]{KoshitaniSpath}), the $G/N$-sets $\IBr(B)$ and $\IBr(b)$ are isomorphic. Therefore, we may assume that $D\unlhd G$. 
Let $e\mid p-1$ be the inertial index of $B$.
By Külshammer~\cite{Kuelshammer} (see also \cite[Theorem~1.19]{habil}), $B$ is Morita equivalent to the group algebra of $D\rtimes Z_e$. In particular, the decomposition matrix of $B$ is given by
\begin{equation}\label{decomp}
\begin{pmatrix}
1&0&\cdots&0\\
0&\ddots&\ddots&\vdots\\
\vdots&\ddots&\ddots&0\\
0&\cdots&0&1\\
1&\cdots&\cdots&1\\
\vdots&&&\vdots\\
1&\cdots&\cdots&1
\end{pmatrix}
\end{equation}
where the first rows correspond to the non-exceptional characters $\chi_1,\ldots,\chi_e\in\Irr(B)$ (see for instance, \cite[Theorem~VII.2.12]{Feit}). 
Since $|G/N|$ is a prime, there exists a $\phi\in\IBr(B)$ such that $\phi^G\in\IBr(G)$. This shows that $B$ is covered by a unique block $B_G$ of $G$ (see \cite[Theorem~9.4]{Navarro}). 
Since $G/N$ is $p'$-group, $B_G$ also has defect group $D$.
Let $e=rq+s$ where $s$ is the number of $G$-invariant Brauer characters of $B$. Then Clifford theory gives $l(B_G)=r+sq$. Now we study the relation between $k(B)=e+\frac{|D|-1}{e}$ and $k(B_G)$. If $G=N\C_G(D)$, then, by considering the generalized decomposition numbers, it is easy to see that $G$ acts trivially on the exceptional characters of $B$. Hence, $k(B_G)=r+sq+q\frac{|D|-1}{e}$. On the other hand, we have $k(B_G)=l(B_G)+\frac{|D|-1}{l(B_G)}$. 
This implies $rq+s=e=q(r+sq)$ and we derive the contradiction $s=0$. We conclude that $\C_G(D)\le N$. Here $G/N$ acts semiregularly on the exceptional characters. This gives $k(B_G)=r+sq+\frac{|D|-1}{eq}$, $(rq+s)q=eq=r+sq$ and $r=0$. Again a contradiction. The proves the first claim. The second claim is obvious, since the trivial character is always $G$-invariant.
\end{proof}

In the situation of \autoref{halftrans}, it is clear that an appropriate quotient $G/M$ with $N\le M$ acts semiregularly on $\IBr(B)$. The action on $\IBr(B)$ induces a graph automorphism on the Brauer tree of $B$. Therefore, a non-trivial action greatly restricts the possible Brauer trees.
We will see later in \autoref{metacyclic} that all possible $\frac{1}{2}$-transitive actions actually occur.

\begin{Lemma}\label{lembasic}
Let $G$ be a finite group with $N\unlhd G$ such that $G/N$ is cyclic. Let $B$ be a $G$-stable block of $N$ with cyclic defect group $D$. Then there exists a basic set $\mathcal{B}$ for $B$ which is $G/N$-isomorphic to $\IBr(B)$ and the Cartan matrix of $B$ with respect to $\mathcal{B}$ is given by $(m+\delta_{ij})$ with $m=\frac{|D|-1}{l(B)}$.
\end{Lemma}
\begin{proof}
Let $Q=(q_{ij})$ be the decomposition matrix of $B$ (with respect to $\IBr(B)$). We may assume that the first $l:=l(B)$ rows correspond to the non-exceptional characters $\chi_1,\ldots,\chi_l$ while the last $m$ rows correspond to the exceptional characters and therefore are all equal.
Moreover, we order the non-exceptional characters and the Brauer characters $\phi_1,\ldots,\phi_l$ in such a way that $Q$ has unitriangular shape (pick leaves from the Brauer tree successively).
Let $S:=(q_{ij})_{i,j=1}^l$. Then $\det(S)=1$ and $S\in\GL(l,\mathbb{Z})$. Hence, we may define $\mathcal{B}=\{\widetilde{\phi}_1,\ldots,\widetilde{\phi}_l\}$ with $\widetilde{\phi}_i:=q_{i1}\phi_1+\ldots+q_{il}\phi_l$ for $i=1,\ldots,l$. 

For $g\in G$ we write $^g\phi_i=\phi_{^gi}$. We will show that $^g\chi_i=\chi_{^gi}$. This is obvious if $i=1$, because $\chi_1$ is a lift of $\phi_1$. Now assume that we have shown the claim for all $k<i$. We have $({^g\chi_i})_{|G_{p'}}=q_{i1}\phi_{^g1}+\ldots+q_{ii}\phi_{^gi}$ where $G_{p'}$ is the set of $p$-regular elements of $G$. By way of contradiction, suppose that $^gi<{^gj}$ for some $j<i$ with $q_{ij}\ne0$. Then it follows that $S$ has two rows of the form
\[\begin{pmatrix}
*&\cdots&*&1&0&\cdots&0\\
*&\cdots&*&1&0&\cdots&0\\
\end{pmatrix}.\]
Hence, $Q$ is not unitriangular and we have a contradiction. Therefore, $^g\chi_i=\chi_{^gi}$ for $i=1,\ldots,l$.

We conclude that
\[^g\widetilde{\phi}_i=q_{i1}\phi_{^g1}+\ldots+q_{il}\phi_{^gl}=(\chi_{^gi})_{|G_{p'}}=q_{^gi,1}\phi_1+\ldots+q_{^gi,l}\phi_l=\widetilde{\phi_{^gi}}.\]
Hence, $\IBr(B)$ and $\mathcal{B}$ are $G/N$-isomorphic.
Let $\mathcal{Q}$ be the decomposition matrix of $B$ with respect to $\mathcal{B}$. Then the first $l$ rows of $\mathcal{Q}$ form an identity matrix.
On the other hand, $B$ is perfectly isometric to the group algebra of $D\rtimes Z_l$ (see for example \cite{Rouquiercyclic}). This implies that $\mathcal{Q}$ is given by \eqref{decomp} above. Form that, the last claim follows easily.
\end{proof}

Next we prove a variant of \autoref{semidirect}.

\begin{Lemma}\label{metacyclic}
Let $p$ be an odd prime, $n_1,n_2\in\mathbb{N}$, $l_1,l_2\mid p-1$, $d\mid(l_1,l_2)$, and let $\gamma_i\in(\mathbb{Z}/p^{n_i}\mathbb{Z})^\times$ of order $l_i$ for $i=1,2$. Let $B$ be a block of 
\begin{align*}
G:=\langle u,v,x,y\mid\,& u^{p^{n_1}}=v^{p^{n_2}}=[u,v]=x^{dl_1}=y^{l_2}=[u,y]=[v,x]=1,\\
&xux^{-1}=u^{\gamma_1},\ yvy^{-1}=v^{\gamma_2},\ yxy^{-1}=x^{1+l_1}\rangle
\end{align*}
lying over a generator $\lambda$ of $\Irr(\langle x^{l_1}\rangle)$. Let $(u,b_u)$ be a $B$-subsection. Then $l(b_u)=l_2$ and the generalized decomposition matrix $Q_u=(d^u_{\chi\phi}:\chi\in\Irr(B),\ \phi\in\IBr(b_u))$ is given as follows

\[Q_u=\begin{pmatrix}
\overmat{d}{1&\cdots&1}{\tr_d(\zeta_1^{\gamma_1})&\cdots&m}\\
\vdots&&\vdots\\
1&\cdots&1\\
&&&\ddots\\
&&&&1&\cdots&1\\
&&&&\vdots&&\vdots\\
&&&&1&\cdots&1\\
\tr_d(\zeta_1^{\gamma_1})&\cdots&\tr_d\bigl(\zeta_1^{\gamma_1^d}\bigr)\\
\vdots&&\vdots\\
\tr_d(\zeta_m^{\gamma_1})&\cdots&\tr_d\bigl(\zeta_m^{\gamma_1^d}\bigr)\\
&\vdots\\
\tr_d\bigl(\zeta_1^{\gamma_1^d}\bigr)&\cdots&\tr_d\bigl(\zeta_1^{\gamma_1^{d^2}}\bigr)\\
\vdots&&\vdots\\
\tr_d\bigl(\zeta_m^{\gamma_1^d}\bigr)&\cdots&\tr_d\bigl(\zeta_m^{\gamma_1^{d^2}}\bigr)\\
&&&\ddots\\
&&&&\tr_d(\zeta_1^{\gamma_1})&\cdots&\tr_d\bigl(\zeta_1^{\gamma_1^d}\bigr)\\
&&&&\vdots&&\vdots\\
&&&&\tr_d(\zeta_m^{\gamma_1})&\cdots&\tr_d\bigl(\zeta_m^{\gamma_1^d}\bigr)\\
&&&&&\vdots\\
&&&&\tr_d\bigl(\zeta_1^{\gamma_1^d}\bigr)&\cdots&\tr_d\bigl(\zeta_1^{\gamma_1^{d^2}}\bigr)\\
&&&&\vdots&&\vdots\\
&&&&\tr_d\bigl(\zeta_m^{\gamma_1^d}\bigr)&\cdots&\tr_d\bigl(\zeta_m^{\gamma_1^{d^2}}\bigr)\\[2mm]
1&\cdots&\cdots&\cdots&\cdots&\cdots&1\\
\vdots&&&&&&\vdots\\
1&\cdots&\cdots&\cdots&\cdots&\cdots&1\\[1mm]
\tr(\zeta_1)&\cdots&\cdots&\cdots&\cdots&\cdots&\tr(\zeta_1)\\
\vdots&&&&&&\vdots\\
\tr(\zeta_1)&\cdots&\cdots&\cdots&\cdots&\cdots&\tr(\zeta_1)\\[1mm]
\tr(\zeta_2)&\cdots&\cdots&\cdots&\cdots&\cdots&\tr(\zeta_2)\\
\vdots&&&&&&\vdots\\
\tr(\zeta_2)&\cdots&\cdots&\cdots&\cdots&\cdots&\tr(\zeta_2)\\
&&&\vdots&&&\\
\tr(\zeta_m)&\cdots&\cdots&\cdots&\cdots&\cdots&\tr(\zeta_m)\\
\vdots&&&&&&\vdots\\
\tr(\zeta_m)&\cdots&\cdots&\cdots&\cdots&\cdots&\tr(\zeta_m)
\end{pmatrix}
\begin{aligned}
&\hspace{-2mm}\left.\begin{matrix}
\vphantom{1}\\
\vphantom{\vdots}\\
\vphantom{1}
\end{matrix}\right\}\frac{l_1}{d}\\[20mm]
&\hspace{-2mm}\left.\begin{matrix}
\vphantom{\tr_d\bigl(\zeta_1^{\gamma_1^d}\bigr)}\\
\vphantom{\vdots}\\
\vphantom{\tr_d\bigl(\zeta_m^{\gamma_1^d}\bigr)}
\end{matrix}\right\}m=\frac{p^{n_1}-1}{l_1}\\[75mm]
&\hspace{-2mm}\left.\begin{matrix}
\vphantom{1}\\
\vphantom{\vdots}\\
\vphantom{1}
\end{matrix}\right\}l_1\frac{p^{n_2}-1}{l_2}\\ 
&\hspace{-2mm}\left.\begin{matrix}
\vphantom{\tr(\zeta_1)}\\
\vphantom{\vdots}\\
\vphantom{\tr(\zeta_1)}
\end{matrix}\right\}\frac{p^{n_2}-1}{l_2}\\[37mm]
\end{aligned}
\]
where $\zeta_1,\ldots,\zeta_m$ is a set of representatives of the non-trivial $p^{n_1}$-th roots of unity under the action of $\gamma_1$, $\tr(\zeta_i)=\sum_{j=1}^{l_1}{\zeta_i^{\gamma_1^j}}$ and $\tr_d(\zeta_i)=\sum_{j=1}^{l_1/d}{\zeta_i^{\gamma_1^{jd}}}$.
\end{Lemma}
\begin{proof}
Since $x^{l_1}\le\Z(G)$, $\lambda$ is $G$-invariant. Hence, by the extended first main theorem (see \cite[Theorem~9.7]{Navarro}) the characters of $\Irr(B)$ lying over $\lambda$ indeed form a block. 
Observe that $\C_G(u)=\langle u,v,x^{l_1},y\rangle$. Hence, $\IBr(b_u)=\{\lambda\rho:\rho\in\Irr(\langle y\rangle)\}$ and $l(b_u)=l_2$. Since $\lambda$ is $G$-invariant, so is $b_u$. 
Now we will use Clifford theory to construct $\Irr(B)$. We will see that all the ramification indices of characters in $\Irr(B)$ are trivial.
Since $^x(\lambda\rho)(y)=(\lambda\rho)(x^{l_1}y)=\lambda(x^{l_1})\rho(y)$, it follows that the orbits of $\IBr(b_u)$ under $G$ all have length $d$. Every character $\lambda\rho\in\IBr(b_u)$ extends in $l_1/d$ many ways to $\langle x^d,y\rangle$. Every such extension induces to $\Irr(B)\cup\IBr(B)$ (considered as inflation from $G/\langle u,v\rangle$). This gives $l_1l_2/d^2$ irreducible (Brauer) characters of $B$. Let $\chi\in\Irr(B)$ be one of them, and let $s\in\langle x^{l_1},y\rangle$. Then $\chi(us)=\chi(s)=(\lambda\rho_1)(s)+\ldots+(\lambda\rho_d)(s)$ where $\lambda\rho_1,\ldots,\lambda\rho_d\in\IBr(b_u)$ is an orbit under $G$. Hence, the corresponding lines of $Q_u$ have the form $(1,\ldots,1,0,\ldots,0)$. 

Now we compute the characters lying over $\theta\lambda\in\Irr(\langle u,v,x^{l_1}\rangle)$ where $1\ne\theta\in\Irr(\langle u\rangle)$. Obviously, $\theta\lambda$ extends in $l_2$ many ways to $\langle u,v,x^{l_1},y\rangle$. Each such extension induces to $\Irr(B)$. This gives $l_2$ new characters. 
For one of them we have 
\[\chi(us)=(\theta_1(u)+\ldots+\theta_{l_1/d}(u))(\lambda\rho_1)(s)+\ldots+(\theta_{l_1-l_1/d+1}(u)+\ldots+\theta_{l_1}(u))(\lambda\rho_d)(s)=\sum_{j=1}^d{\tr_d\bigl(\zeta_i^{\gamma_1^j}\bigr)(\lambda\rho_j)(s)}\]
for some $i$. We have $(p^{n_1}-1)/l_1$ essentially different choices for $\theta$. Thus in total we get $l_2(p^{n_1}-1)/l_1$ characters in this way.

Next consider $\theta\lambda\in\Irr(\langle u,v,x^{l_1}\rangle)$ where $1\ne\theta\in\Irr(\langle v\rangle)$. This character extends in $l_1$ many ways to $\langle u,v,x\rangle$ and the extensions induce to $\Irr(B)$. In total there are $l_1(p^{n_2}-1)/l_2$ such characters $\chi$. Since $\langle u,v,x^{l_1}\rangle\unlhd\langle u,v,x^{l_1},y\rangle$, it is easy to see that $\chi(us)=\chi(s)=\bigl(\sum_{i=1}^{l_2}\lambda\rho_i\bigr)(s)$.

Finally, let $1\ne\theta\in\Irr(\langle u\rangle)$ and $1\ne\eta\in\Irr(\langle v\rangle)$. Then the characters $\theta\eta\lambda$ induce directly to $\Irr(B)$. If $\chi$ is one of them, we get \[\chi(us)=\Bigl(\sum_{i=1}^{l_1}{^{x^i}(\theta\eta\lambda)^{\langle u,v,x^{l_1},y\rangle}}\Bigr)(us)=\sum_{i=1}^{l_1}{\theta_i(u)}\sum_{j=1}^{l_2}{(\lambda\rho_i)(s)}=\tr(\zeta_i)\sum_{j=1}^{l_2}{(\lambda\rho_i)(s)}.\]
It is easy to see that this amounts all the characters in $\Irr(B)$. 
\end{proof}

\begin{Lemma}\label{finallem}
Let $Q_u$ be the matrix from \autoref{metacyclic} with columns $q_i$. Then we can write $q_i$ with respect to the integral basis $\widetilde{T}_n$ from \autoref{intbas} with $n=n_1$ and $|\mathcal{N}|=l_1/d$. Hence, let $A_i\in\mathbb{Z}^{k\times p^{n-1}(p-1)d/l_1}$ such that $q_i=A_i\widetilde{T}_n$. 
Let $A=(A_{ij})_{i,j=1}^{l_2/d}$ the block matrix with blocks 
\[A_{ij}=\begin{cases}
A_1^\textup{T}A_1&\text{if }i=j,\\
A_1^\textup{T}A_{d+1}&\text{if }i\ne j.
\end{cases}\]
Let $U\in\mathbb{Z}^{r\times s}$ without vanishing rows such that $U^\textup{T}U=A$. Then 
\[r\le\frac{p^{n_1}-1}{l_1}\frac{p^{n_2}-1}{l_2}+l_1\frac{p^{n_2}-1}{l_2}+l_2\frac{p^{n_1}-1}{l_1}+\frac{l_1l_2}{d^2}.\]
\end{Lemma}
\begin{proof}
First observe that $1=-\sum_{t\in p^{n-1}T_1}{\tr_d(\zeta^{t})}$ where $\zeta\in\mathbb{C}$ is a primitive $p^n$-th root of unity and $\tr_d$ as in \autoref{metacyclic}. Thus, the first rows of $A_1$ have the form $(-1,\ldots,-1,0,\ldots,0)$ with $(p-1)d/l_1$ entries $-1$. 
Now we use induction on $n$ to choose the elements $\zeta_i$ from \autoref{metacyclic} in such a way that $A_1$ has a nice shape. If $n=1$, then all the entries $\tr_d\bigl(\zeta_i^{\gamma_1^j}\bigr)$ in $q_1$ are just elements of $\widetilde{T}_n$. So the corresponding rows of $A_1$ consist of zeros and exactly one entry $1$. Now let $n\ge 2$. Then, by induction, the entries $\tr_d\bigl(\zeta_i^{\gamma_1^j}\bigr)$ where $\zeta_i$ is not primitive form a part $A_1^{(p)}$ of $A_1$ which we already know. There are $\frac{p^n-1}{l_1}-\frac{p^{n-1}-1}{l_1}=p^{n-1}\frac{p-1}{l_1}$ remaining entries (counting without multiplicities). We choose these $\zeta_i$ such that they form blocks $\{\zeta^{s+jp^{n-1}}:j=0,\ldots,p-1\}$ for some $s\in S_{n-1}$ with the notation from \autoref{intbas}. It is crucial here to see that these elements are not conjugate under $\gamma_1$. Moreover, only the elements $\zeta^{s+(p-1)p^{n-1}}$ do not lie in $\widetilde{T}_n$. Here however, $\zeta^{s+(p-1)p^{n-1}}=-\sum_{j=0}^{p-2}{\zeta^{s+jp^{n-1}}}$. Hence, the rows in $A_1$ corresponding to $\tr_d\bigl(\zeta_1^{\gamma_1^j}\bigr),\ldots,\tr_d\bigl(\zeta_m^{\gamma_1^j}\bigr)$ for a fixed $j$ have the following form:
\[\begin{pmatrix}
A_1^{(p)}\\
&1\\
\\
&\multicolumn{10}{c}{\diagdots[-13]{19em}{.4em}}\\
\\
&&&&&&&&&&1\\[1mm]
&-1&\cdots&-1\\
&&&&-1&\cdots&-1&&&&&0\\
&&&&&&&\ddots\\
&&&&&&&&-1&\cdots&-1
\end{pmatrix}\begin{aligned}
&\\[-1mm]
&\hspace{-2mm}\left.\begin{matrix}
\vphantom{1}\\
\\
\multicolumn{10}{c}{\vphantom{\diagdots[-13]{19em}{.4em}}}\\
\\
\vphantom{1}
\end{matrix}\right\}p^{n-2}\frac{(p-1)^2}{l_1}\\
&\hspace{-2mm}\left.\begin{matrix}
\vphantom{-1}\\
\vphantom{-1}\\
\vphantom{\vdots}\\
\vphantom{-1}
\end{matrix}\right\}p^{n-2}\frac{p-1}{l_1}\\ 
\end{aligned}\]

It remains to handle the numbers $\tr(\zeta_i)$. But this follows directly from what we have done, since $\tr(\zeta_i)=\sum_{j=1}^d{\tr_d(\zeta_i^{\gamma_1^j})}$. In case $n=1$ the matrix $A_1^\text{T}A_1$ is given as follows

\vspace{4mm}
\[1_{d\frac{p-1}{l_1}}+
\begin{pmatrix}
\overmat{d}{1&\cdots&1}{1&\cdots&1}\\
\vdots&&\vdots\\
1&\cdots&1
\end{pmatrix}
\otimes
\begin{pmatrix}
\overmat{\frac{p-1}{l_1}}{\alpha&\beta&\cdots&\beta}{\alpha&\cdots&\cdots&\beta}\\
\beta&\ddots&\ddots&\vdots\\
\vdots&\ddots&\ddots&\beta\\
\beta&\cdots&\beta&\alpha
\end{pmatrix}\hspace{1cm}
\begin{aligned}
&\alpha:=\frac{l_1}{d}+(l_1+1)\frac{p^{n_2}-1}{l_2},\\
&\beta:=\frac{l_1}{d}+l_1\frac{p^{n_2}-1}{l_2}.
\end{aligned}\]
Now let $n\ge 2$. 
Suppose that the matrix $A_1^{\text{T}}A_1$ for $n-1$ is given by $B_{n-1}$. Then $A_1^\text{T}A_1$ for $n$ has the form

\vspace{4mm}
\[A_1^\text{T}A_1=B_{n-1}\oplus\left(1_{p^{n-2}\frac{p-1}{l_1}}\otimes
\begin{pmatrix}
\overmat{d}{\delta+1&\delta&\cdots&\delta}{\delta+1&\cdots&\cdots&\delta+1}\\
\delta&\ddots&\ddots&\vdots\\
\vdots&\ddots&\ddots&\delta\\
\delta&\cdots&\delta&\delta+1
\end{pmatrix}\otimes
\begin{pmatrix}
\overmat{p-1}{2&1&\cdots&1}{2&\cdots&\cdots&1}\\
1&\ddots&\ddots&\vdots\\
\vdots&\ddots&\ddots&1\\
1&\cdots&1&2
\end{pmatrix}\right)\hspace{1cm}\begin{aligned}
\delta:=\frac{p^{n_2}-1}{l_2}
\end{aligned}
.\]
Here the integral basis vectors in the second summand are ordered in the following way: $s_1$, $s_1+p^{n-1},\ldots,s_1+(p-2)p^{n-1}$, $\gamma_1(s_1),\ldots,\gamma_1(s_1+(p-2)p^{n-1}),\ldots,\gamma_1^{d-1}(s_1+(p-2)p^{n-1})$, $s_2,\ldots$ where $s_i\in S_{n-1}$ with the notation from \autoref{intbas}.

Now assume that $d<l_2$, so that $A_{d+1}$ really exists. In case $n=1$ we have

\vspace{4mm}
\[A_1^\text{T}A_{d+1}=\delta
\begin{pmatrix}
\overmat{d}{1&\cdots&1}{1&\cdots&1}\\
\vdots&&\vdots\\
1&\cdots&1
\end{pmatrix}
\otimes
\begin{pmatrix}
\overmat{\frac{p-1}{l_1}}{l_1+1&l_1&\cdots&l_1}{l_1+1&\cdots&\cdots&l_1+1}\\
l_1&\ddots&\ddots&\vdots\\
\vdots&\ddots&\ddots&l_1\\
l_1&\cdots&l_1&l_1+1
\end{pmatrix}.\]
As before denote the matrix $A_1^\text{T}A_{d+1}$ for $n-1$ by $B_{n-1}'$. Then for $n\ge 2$ we obtain

\vspace{4mm}
\[A_1^\text{T}A_{d+1}=B_{n-1}'\oplus\delta\left(1_{p^{n-2}\frac{p-1}{l_1}}\otimes
\begin{pmatrix}
\overmat{d}{1&\cdots&1}{1&\cdots&1}\\
\vdots&&\vdots\\
1&\cdots&1
\end{pmatrix}\otimes
\begin{pmatrix}
\overmat{p-1}{2&1&\cdots&1}{2&\cdots&\cdots&1}\\
1&\ddots&\ddots&\vdots\\
\vdots&\ddots&\ddots&1\\
1&\cdots&1&2
\end{pmatrix}\right)
.\]
Finally, we consider $A=U^\text{T}U$. In order to estimate $r$, we may replace $A$ by $SAS^\text{T}$ where $S\in\GL(s,\mathbb{Z})$ with $s=p^{n-1}(p-1)l_2/l_1$. In particular, we may permute the rows and columns of $A$, so that $A$ has block diagonal form. The first block corresponds to the integral basis elements $p^{n-1}T_1$ and so on. It is enough to handle each of these blocks separately, i.\,e. we may assume that $U^\text{T}U$ equals one of these blocks. We use induction on $n$. If $n=1$, then there is only one block. Let $y_i$ be the quadratic form corresponding to the Dynkin diagram of type $A_i$ (not the matrix here). By \autoref{lemtensor}, we can use the quadratic form $y_{l_2/d}\otimes y_d\otimes y_{(p-1)/l_1}$ in order to estimate $r$ like in \autoref{stablecase}. For $n=1$ we get
\begin{align*}
r&\le \frac{l_2}{d}\biggl(d\Bigl((\alpha+1)\frac{p-1}{l_1}-\beta\bigl(\frac{p-1}{l_1}-1\bigr)\Bigr)-(d-1)\Bigl(\alpha\frac{p-1}{l_1}-\beta\bigl(\frac{p-1}{l_1}-1\bigr)\Bigr)\biggr)-\Bigl(\frac{l_2}{d}-1\Bigr)\delta\Bigl(\frac{p-1}{l_1}+l_1\Bigr)\\
&=\frac{l_2}{d}\biggl(\alpha\frac{p-1}{l_1}-\beta\Bigl(\frac{p-1}{l_1}-1\Bigr)+d\frac{p-1}{l_1}\biggr)-\Bigl(\frac{l_2}{d}-1\Bigr)\delta\Bigl(\frac{p-1}{l_1}+l_1\Bigr)\qquad(\text{use }\alpha=\beta+\delta)\\
&=\frac{l_2}{d}\biggl(\beta+(\delta+d)\frac{p-1}{l_1}\biggr)-\Bigl(\frac{l_2}{d}-1\Bigr)\delta\Bigl(\frac{p-1}{l_1}+l_1\Bigr)\qquad(\text{use definition of }\beta)\\
&=\frac{l_1l_2}{d^2}+\delta\frac{p-1}{l_1}+l_2\frac{p-1}{l_1}+l_1\delta
\end{align*}
and this is exactly what we wanted to have. Now let $n\ge 2$. By induction the block of $A$ corresponding to the indices $t\in pT_{n-1}$ gives 
\begin{equation}\label{n1}
r\le\frac{p^{n_1-1}-1}{l_1}\frac{p^{n_2}-1}{l_2}+l_1\frac{p^{n_2}-1}{l_2}+l_2\frac{p^{n_1-1}-1}{l_1}+\frac{l_1l_2}{d^2}.
\end{equation}
For the remaining block corresponding to the indices $t\in T_n\setminus pT_{n-1}$ we use the quadratic form $y_{l_2/d}\otimes y_{p^{n-2}(p-1)/l_1}\otimes y_d\otimes y_{p-1}$. This gives
\begin{align*}
r&\le \frac{l_2}{d}\biggl(p^{n-2}\frac{p-1}{l_1}\Bigl(d\bigl((p-1)2(\delta+1)-(p-2)(\delta+1)\bigr)-(d-1)\bigl((p-1)2\delta-(p-2)\delta\bigr)\Bigl)\biggr)\\
&\quad-\Bigl(\frac{l_2}{d}-1\Bigr)\Bigl(p^{n-1}\frac{p-1}{l_1}\delta\Bigr)\\
&=\frac{l_2}{d}p^{n-2}\frac{p-1}{l_1}\bigl(d(\delta+1)p-(d-1)\delta p\bigr)-\Bigl(\frac{l_2}{d}-1\Bigr)\Bigl(p^{n-1}\frac{p-1}{l_1}\delta\Bigr)\\
&=\frac{l_2}{d}p^{n-1}\frac{p-1}{l_1}(d+\delta)-\Bigl(\frac{l_2}{d}-1\Bigr)\Bigl(p^{n-1}\frac{p-1}{l_1}\delta\Bigr)\\
&=p^{n-1}\frac{p-1}{l_1}\delta+l_2p^{n-1}\frac{p-1}{l_1}.
\end{align*}
Adding this to \eqref{n1} gives precisely our claim.
\end{proof}

\begin{proof}[Proof of \autoref{2rank}]
Let $D=\langle u\rangle\times\langle v\rangle$ such that $I$ acts on $\langle u\rangle$ and $J$ acts on $\langle v\rangle$. 
If $D\unlhd G$, then there is nothing to prove. Hence, we may assume that $\langle u\rangle\ntrianglelefteq G$. Let $b_D$ be a Brauer correspondent of $B$ in $\C_G(D)$. By Fong-Reynolds (see \cite[Theorem~9.14]{Navarro}), we may assume that $b$ is a block of $N:=\N_G(D,b_D)$. In particular, $\langle u\rangle,\langle v\rangle\unlhd N$. By Külshammer~\cite{Kuelshammer} (see also \cite[Theorem~1.19]{habil}), $b$ is Morita equivalent to a twisted group algebra $\mathcal{O}_\gamma[D\rtimes I(B)]$ where $\mathcal{O}$ is a suitable complete discrete valuation ring. Let $|I|=l_1$ and $|J|=l_2$.
By the Künneth formula, $\cohom^2(I(B),\mathbb{C}^\times)\cong Z_{(l_1,l_2)}$. 
Hence, we may assume that $b$ is Morita equivalent to the block in \autoref{metacyclic} such that $l(b)=l_1l_2/d^2$ for some $d\mid(l_1,l_2)$ (see \cite[Proposition~1.20]{habil}; this does not depend on the choice of the covering group by \cite[Proposition~5.15]{PuigMod}). 

Let $b_u:=b_D^{\C_G(u)}$ so that $(u,b_u)$ is a major $B$-subsection.
Let $B_u:=b_u^{\N_G(\langle u\rangle,b_u)}$. Then
\[B_u=b_D^{\N_G(\langle u\rangle,b_u)}=b^{\N_G(\langle u\rangle,b_u)}.\]
Observe that $B_u$ has defect group $D$ and inertial quotient $I(B)$. Since $\langle u\rangle\ntrianglelefteq G$, induction on $|G|$ shows that $l(B_u)\le l(b)$. Now $b_u$ dominates a block $\overline{b_u}$ with cyclic defect group $\langle v\rangle$ and inertial index $l_2$ (see \cite[Lemma~3]{SambaleC4}). In particular, $l(b_u)=l_2$. By \autoref{halftrans}, $\N_G(\langle u\rangle,b_u)$ acts $\frac{1}{2}$-transitively on $\IBr(\overline{b_u})=\IBr(b_u)$. Suppose that the orbits have length $\alpha$. Then Clifford theory shows that $l(B_u)=l(b_u)l_1/\alpha^2=l_1l_2/\alpha^2$. Therefore, $\alpha\ge d$. 

Next we show that $k(B)-l(B)=k(b)-l(b)$ by using a formula of Brauer (see \cite[Theorem~1.35]{habil}). Let $1\ne w\in D$ and $b_w:=b_D^{\C_G(w)}$. If $w=u^iv^j$ such that $u^i\ne1$ and $v^j\ne 1$, then $l(b_w)=1$. Now assume that $w\in\langle u\rangle$. As usual, $\overline{b_w}$ has defect group $D/\langle w\rangle$. A result by Watanabe~\cite[Theorem~1]{Watanabe1} (see also \cite[Theorem~1.39]{habil}) implies that $l(b_w)=l(\overline{b_w})=l(\overline{b_u})=l_2$. The case $w\in\langle v\rangle$ is similar. Altogether we have proved that 
\begin{equation}\label{diff}
k(B)-l(B)=\frac{p^n-1}{l_1}\frac{p^m-1}{l_2}+l_1\frac{p^m-1}{l_2}+l_2\frac{p^n-1}{l_1}.
\end{equation}
The same arguments apply to $b$ and it follows that $k(B)-l(B)=k(b)-l(b)$. Hence, it suffices to show that $l(B)\le l(b)$.

In order to apply \autoref{general} we need to know the Cartan matrix $\overline{C_u}$ of $\overline{b_u}$. By \autoref{lembasic} we can choose a basic set such that $\overline{C_u}=(t+\delta_{ij})$ with $t=(p^m-1)/l_2$. 
Since it is still complicated to use the formulas in \autoref{general}, we investigate an alternative approach to compute the matrix $M$ in \autoref{general}. 
Let $\widetilde{B}$ be the block from \autoref{metacyclic} where we choose the parameter $d$ in that lemma to be $\alpha$. Let $(u,\beta_u)$ be the corresponding $\widetilde{B}$-subsection. Then the irreducible Brauer characters of $\beta_u$ also split in orbits of length $\alpha$. Moreover, since $\beta_u$ has a normal defect group, the Cartan matrix of $\overline{\beta_u}$ is the same as $\overline{C_u}$ above. Therefore, the generalized decomposition numbers with respect to $(u,\beta_u)$ satisfy exactly the same refined orthogonality relations as the corresponding numbers with respect to $(u,b_u)$. 
Furthermore, we only need to consider one irreducible Brauer character from each orbit. By the special shape of $\overline{C_u}$, the scalar product of two columns $d_i$ and $d_j$ of the generalized decomposition matrix only depends on whether $i=j$ or $i\ne j$.
Hence, the matrix $M$ in \autoref{general} is equivalent to the matrix $A$ we have computed in \autoref{finallem}. It follows that $l(B)\le l_1l_2/\alpha^2\le l(b)$ and we are done.
\end{proof}

\section{A conjecture by Navarro}

Very recently Gabriel Navarro~\cite{NavarrokB} proposed the following conjecture as a generalization of Brauer's $k(B)$-Conjecture for principal blocks.

\begin{Conjecture}
Let $G$ be a finite group with $N\unlhd G$, and let $p$ be a prime. Let $B_G$ and $B_{G/N}$ be the principal $p$-blocks of $G$ and $G/N$ respectively. Then $k(B_G)\le k(B_{G/N})|N|_p$.
\end{Conjecture}

Here, as usual $|N|_p$ denotes the $p$-part of the order of $N$.
We use \autoref{stablecase} above to prove this conjecture in a special case.

\begin{Theorem}\label{navarro}
Let $G$ be a finite group with $N\unlhd G$ such that the Sylow $p$-subgroups of $N$ and $G/N$ are cyclic for a prime $p$.
Then for the principal blocks we have $k(B_G)\le k(B_{G/N})|N|_p$.
\end{Theorem}
\begin{proof}
Let $P$ be a Sylow $p$-subgroup of $G$. Then $P\cap N$ and $P/P\cap N\cong PN/N$ are Sylow subgroups of $N$ and $G/N$ respectively. By hypothesis, these groups are cyclic and therefore $P$ is metacyclic. If $B_G$ is nilpotent, then $G$ is $p$-nilpotent and $B_{G/N}$ must be nilpotent too. In this case the well-known theorem by Broué-Puig~\cite{BrouePuig} says that $k(B_G)=k(P)\le k(P/P\cap N)|P\cap N|=k(B_{G/N})|N|_p$. Thus, we may assume that $B$ is non-nilpotent. 

Now we deal with the case where $P$ is non-abelian. If $p=2$, then \cite[Theorem~8.1]{habil} shows that $P$ is a dihedral group, a semidihedral group or a quaternion group. This implies that $|P/P\cap N|=2$ and $k(B_G)\le |P|=k(B_{G/N})|N|_p$ follows from \cite[Corollary~8.2]{habil}. Hence, let $p>2$. Here \cite[Theorem~8.8]{habil} shows that $P$ is a split extension of two cyclic groups. From the proof of this result, we see further that $\N_G(P)$ acts non-trivially on $P'$. Hence, $\N_G(P)$ acts non-trivially on $P\cap N$ as well. Moreover, the focal subgroup $G'\cap P$ is cyclic. This implies that $\N_G(P)$ acts trivially on $P/P\cap N$ (see \cite[proof of Theorem~8.8]{habil}). 
Since $\N_G(PN/N)$ permutes the Sylow $p$-subgroups of $PN$, the Frattini argument implies that $\N_G(PN/N)=\N_G(P)N$. Therefore, $\N_G(PN/N)$ acts trivially on $PN/N$ and $B_{G/N}$ is nilpotent. Again the claim follows from $k(B_G)\le|P|$.

Finally, it remains to handle the abelian groups $P$. We may assume that $1<P\cap N<P$. Since $\N_G(P)$ normalizes $P\cap N$, it also normalizes a non-trivial subgroup of $\Omega(P)=\{x\in P:x^p=1\}$. By Maschke's Theorem, $\Omega(P)$ decomposes into a direct sum of $\N_G(P)$-invariant cyclic subgroups. We show that this is also true for $P$. By \cite[Theorem~5.2.2]{Gorenstein}, we may assume that $P$ is a direct product of two isomorphic cyclic groups in order to show this claim. Choose an element $x\in P$ of maximal order such that $\Omega(\langle x\rangle)$ is $\N_G(P)$-invariant. Suppose that there exists $g\in\N_G(P)$ such that $gxg^{-1}\notin\langle x\rangle$. Then Burnside's Basis Theorem shows $P=\langle x,gxg^{-1}\rangle$ and we derive the contradiction 
\[\Omega(\langle x\rangle)=\Omega(\langle x\rangle)\cap g\Omega(\langle x\rangle)g^{-1}=\Omega(\langle x\rangle)\cap\Omega(\langle gxg^{-1}\rangle)=\Omega(\langle x\rangle\cap\langle gxg^{-1}\rangle)=1.\]
Hence, we may assume that $\langle x\rangle$ is $\N_G(P)$-invariant. By \cite[Theorem~3.3.2]{Gorenstein}, $\langle x\rangle$ contains an $\N_G(P)$-invariant complement in $P$.

Thus, let $P=\langle u\rangle\times\langle v\rangle$ with $\N_G(P)$-invariant subgroups $\langle u\rangle$ and $\langle v\rangle$ (not necessarily of the same order). Let $P\cap N=\langle u^iv^j\rangle$ for some $i,j\in\mathbb{Z}$. Since $P/P\cap N$ is cyclic, so is $P/\langle u^i,v^j\rangle$. Thus, we may assume that $j=1$ without loss of generality. Let $l=\lvert\N_G(P)/\C_G(u)\rvert$. Then $l$ is the inertial index of $B_{G/N}$. Hence, by Dade~\cite{Dade} (see also \cite[Theorem~8.6]{habil}), we have $k(B_{G/N})=\frac{|P/P\cap N|-1}{l}+l$. Now let $b_u$ be the principal block of $\C_G(u)$, so that $(u,b_u)$ is a $B_G$-subsection. Then $b_u$ dominates the principal block $\overline{b_u}$ of $\C_G(u)/\langle u\rangle$ with cyclic defect group $P/\langle u\rangle\cong\langle v\rangle$. By \autoref{halftrans}, all irreducible Brauer characters of $b_u$ are stable under $\N_G(\langle u\rangle)$. Hence, \autoref{stablecase} (applied with the now familiar quadratic form corresponding to the Dynkin diagram of type $A$) implies
\[k(B_G)\le\Bigl(\frac{|\langle u\rangle|-1}{l}+l\Bigr)|\langle v\rangle|.\]
It suffices to show that
\[\Bigl(\frac{|\langle u\rangle|-1}{l}+l\Bigr)|\langle v\rangle|\le \Bigl(\frac{|P/P\cap N|-1}{l}+l\Bigr)|P\cap N|.\]
This reduces easily to $|\langle v\rangle|\le|P\cap N|$ which is true. The claim follows.
\end{proof}

The arguments of the proof of \autoref{navarro} can be used to verify the conjecture for many groups with metacyclic Sylow $p$-subgroups $P$. However, at the moment not all of these groups can be handled. For instance, image that $P\cong Z_{p^2}^2$ and $P\cap N\cong Z_p^2$. In these cases, the claim would follow from Alperin's Weight Conjecture. 
By \cite[Section~13.1]{habil}, it is not hard to prove the conjecture in case $p=2$ and $|P|\le 16$. We have also used GAP~\cite{GAP47} to check the conjecture for all groups in the libraries of small groups and perfect groups. 

\section{Concluding remarks}

In view of \autoref{examp} and \autoref{2rank}, it appears that the inequalities for $k(B)$ get better if not all irreducible Brauer characters are stable. This seems to be related to the fact that the number of simple modules of a twisted group algebra is not larger than the number for the corresponding untwisted group algebra.
Consequently, one might hope that \autoref{stablecase} also holds without the hypothesis on the Brauer characters. It would then be a direct generalization of both \autoref{outer} and \autoref{mainfusion}. 
The case $p=2$ has already been done in \cite[Theorem~5.2]{habil}. The odd case seems much harder to prove and we did not find a counterexample either, but we can provide a partial answer which generalizes \cite[Proposition~5.13]{habil}.

\begin{Proposition}
Let $B$ be a $p$-block of a finite group $G$. Let $(u,b_u)$ be a major $B$-subsection such that $b_u$ dominates a block $\overline{b_u}$ of $\C_G(u)/\langle u\rangle$. Suppose that $l(b_u)=2$, and let $\overline{C_u}=(c_{ij})$ be the Cartan matrix of $\overline{b_u}$ up to basic sets. Let $r:=\lvert\N_G(\langle u\rangle,b_u)/\C_G(u)\rvert$. Then
\[k(B)\le\Bigl(\frac{|\langle u\rangle|-1}{r}+r\Bigr)\bigl(c_{11}+c_{22}-c_{12}\bigr).\]
\end{Proposition}
\begin{proof}
By \autoref{stablecase}, we may assume that $\N_G(\langle u\rangle,b_u)$ interchanges the two irreducible Brauer characters of $b_u$. 
We suppose first that $\overline{C_u}$ is the actual Cartan matrix (not up to basic sets). Then it follows that $c_{11}=c_{22}$. Since the determinant of $\overline{C_u}$ is a $p$-power, we have
\[(c_{11}-c_{12})(c_{11}+c_{12})=c_{11}c_{22}-c_{12}^2=\det(\overline{C_u})=p^d\]
for some $d\in\mathbb{N}$. Let $c_{11}-c_{12}=p^{d_1}$ and $c_{11}+c_{12}=p^{d_2}$. Then $c_{11}=(p^{d_1}+p^{d_2})/2$ and $c_{12}=(p^{d_2}-p^{d_1})/2$. Setting $\alpha:=(p^{d_2-d_1}-1)/2$ we have
\[\overline{C_u}=p^{d_1}\begin{pmatrix}
\alpha+1&\alpha\\\alpha&\alpha+1
\end{pmatrix}.\]
Hence, $p^{-d_1}\overline{C_u}$ is the Cartan matrix of the principal block of $Z_{p^{d_2-d_1}}\rtimes Z_2$. So we are basically in the situation of \autoref{metacyclic} and \autoref{general} implies
\begin{align}\label{l2better}
k(B)&\le p^{d_1}\Bigl(\frac{|\langle u\rangle|-1}{r}\frac{p^{d_2-d_1}-1}{2}+2\frac{|\langle u\rangle|-1}{r}+r\frac{p^{d_2-d_1}-1}{2}+\frac{r}{2}\Bigr)\nonumber\\
&= p^{d_1}\Bigl(\frac{|\langle u\rangle|-1}{r}\frac{p^{d_2-d_1}+3}{2}+\frac{rp^{d_2-d_1}}{2}\Bigr)<\Bigl(\frac{|\langle u\rangle|-1}{r}+r\Bigr)\bigl(c_{11}+c_{22}-c_{12}\bigr)
\end{align}
(cf. \autoref{finallem}, we omit the details). Note that $r$ must be even. Now it remains to show that the right hand side of \eqref{l2better} does not get smaller if we change the basic set. For
\[U:=\begin{pmatrix}
1&0\\
\vdots&\vdots\\
1&0\\
1&1\\
\vdots&\vdots\\
1&1\\
0&1\\
\vdots&\vdots\\
0&1
\end{pmatrix}\begin{aligned}
&\hspace{-2mm}\left.\begin{matrix}
\vphantom{1}\\[1mm]
\vphantom{1}\\
\vphantom{1}
\end{matrix}\right\}c_{11}-c_{12}\\
&\hspace{-2mm}\left.\begin{matrix}
\vphantom{1}\\[1mm]
\vphantom{1}\\
\vphantom{1}
\end{matrix}\right\}c_{12}\\
&\hspace{-2mm}\left.\begin{matrix}
\vphantom{1}\\[1mm]
\vphantom{1}\\
\vphantom{1}
\end{matrix}\right\}c_{22}-c_{12}\\
\end{aligned}\]
we have $U^\text{T}U=\overline{C_u}$ and $U$ has just $c_{11}+c_{22}-c_{12}$ rows. Suppose that $S\in\GL(2,\mathbb{Z})$, and let $S^\text{T}\overline{C_u}S=(c'_{ij})$. Then the matrix $US$ still has $c_{11}+c_{22}-c_{12}$ (non-zero) rows and satisfies $(US)^\text{T}(US)=(c'_{ij})$. An application of the quadratic form corresponding to the Dynkin diagram of type $A_2$ (as in \autoref{stablecase}) shows that $c_{11}+c_{22}-c_{12}\le c'_{11}+c'_{22}-c'_{12}$ and we are done.
\end{proof}

In praxis, if $\overline{C_u}$ and $\mathcal{N}$ are known and the action on $\IBr(b_u)$ is unknown, one can consider all possible actions, apply \autoref{general} in each case and take the bound from the worst case.

We finish the paper with some more remarks:
\begin{enumerate}[(i)]
\item The idea of \autoref{2rank} goes back to Kiyota~\cite{Kiyota}. On page~40 he considers the case $p=3$ and $D\rtimes I(B)\cong\mathfrak{S}_3^2$ where $b_u$ has two irreducible Brauer characters which are interchanged by $\N_G(\langle u\rangle,b_u)$. Using the integral basis $1$, $e^{2\pi i/3}$ he obtains the matrix $\bigl(\begin{smallmatrix}5&1\\1&2
\end{smallmatrix}\bigr)$. In comparison, we use the integral basis $e^{2\pi i/3}$, $e^{-2\pi i/3}$ and get the equivalent matrix $\bigl(\begin{smallmatrix}5&4\\4&5
\end{smallmatrix}\bigr)$ in \autoref{finallem}. This is, by the way, the Cartan matrix of the generalized dihedral group $Z_3^2\rtimes Z_2$. It is a bit harder to find such non-trivial examples for $p=2$. Gluck~[Question~1 and Section~2]\cite{Gluck} gave an example with $|D|=32$.

\item Whenever we have a major $B$-subsection $(u,b_u)$ such that $\overline{b_u}$ has cyclic defect groups, it is obvious that the defect group $D$ of $B$ is abelian of rank at most $2$. Hence, \autoref{2rank} covers a large portion of these cases.
However, in principle it is possible to do similar things for non-major subsections. The details are necessarily more complicated. In particular, one needs to combine \autoref{semidirect} and \autoref{metacyclic}. Since we do not know any interesting applications, we chose to stick to major subsections.

\item Brauer's permutation lemma (see \cite[Lemma~IV.6.10]{Feit}) asserts in most cases that the number of integral columns of the generalized decomposition matrix of a block $B$ coincides with the number of $p$-rational irreducible characters of $B$. Thus, if we have some knowledge about the fusion system of $B$, we can obtain lower bounds for the number of integral decomposition numbers $d^u_{\chi\phi}$ where $(u,b_u)$ is a fixed subsection. This has consequences for the possible factorizations $M=Q^\text{T}Q$ in \autoref{general}. Additional information can be taken into account by using the Broué-Puig $*$-construction in order to compare the contributions of distinct subsections. We have successfully applied these arguments in \cite[Lemma~1]{Sambalerank3}.

\item Judging from the examples we have considered, it seems that the non-zero elementary divisors of the matrix $M$ in \autoref{general} are always $p$-powers. We do not know if these numbers have interesting interpretations. At least, $M$ does not seem to be related to the Cartan matrix of $B_u=b_u^{\N_G(\langle u\rangle,b_u)}$ (see \autoref{examp}). Moreover, none of the inequalities 
\begin{align*}
k(B)\le k(b_u),&& k(B)\le|\langle u\rangle|k(\overline{b_u}),&& k(B)\le k(B_u)
\end{align*} 
is true in general. 
Counterexamples are given by the principal $3$-blocks of the groups $\texttt{SmallGroup}(108,17)$, $\texttt{SmallGroup}(324,117)$ and $\mathfrak{S}_3^2$ respectively (see \cite{GAP47}).

\item The Cartan matrix of (the principal block of) $G=\mathfrak{A}_4^2$ is given by $C=(c_{ij})=(1+\delta_{ij})_{i,j=1}^3\otimes(1+\delta_{ij})_{i,j=1}^3$. We have shown in \cite[end of Section~4.1]{habil} that there is no positive definite, integral quadratic form $q(x)=\sum_{1\le i\le j\le9}{q_{ij}x_ix_j}$ such that
\begin{equation}\label{A42}
\sum_{1\le i\le j\le 9}{q_{ij}c_{ij}}=16=k(G).
\end{equation}
Now let $q$ be the tensor product of the quadratic form corresponding to the Dynkin diagram of type $A_3$ with itself. Then $q$ is not integral anymore, but by \autoref{lemtensor}, $q$ can be used to bound $k(B)$ and it turns out that \eqref{A42} holds. Therefore, we really benefit from giving up integrity. In fact, we do not know any Cartan matrix where this method fails to give Brauer's $k(B)$-Conjecture $k(B)\le|D|$.


\item Thinking more general, one can try to replace subsections by subpairs $(Q,b_Q)$ where $Q\le D$ and $b_Q$ is a Brauer correspondent of $B$ in $\C_G(Q)Q$. This seems a bit out of reach for the following reason. The block $b_D$ is nilpotent and has Cartan matrix $(|D|)$. Hence, a general version of \autoref{outer} would already imply Brauer's $k(B)$-Conjecture for all blocks.
\end{enumerate}

\section*{Acknowledgment}
This work is supported by the German Research Foundation (project SA 2864/1-1) and the Daimler and Benz Foundation (project 32-08/13). The author thanks Gabriel Navarro for the hospitality at the University of Valencia and for showing him his conjecture. The author is also grateful to Britta Späth and Burkhard Külshammer for answering some questions concerning \autoref{halftrans}.


\begin{thebibliography}{10}

\bibitem{AKO}
M. Aschbacher, R. Kessar and B. Oliver, \textit{Fusion systems in algebra and
  topology}, London Mathematical Society Lecture Note Series, Vol. 391,
  Cambridge University Press, Cambridge, 2011.

\bibitem{BrauerBlSec2}
R. Brauer, \textit{On blocks and sections in finite groups. {II}}, Amer. J.
  Math. \textbf{90} (1968), 895--925.

\bibitem{BroueSanta}
M. Broué, \textit{On characters of height zero}, in: The {S}anta {C}ruz
  {C}onference on {F}inite {G}roups ({U}niv. {C}alifornia, {S}anta {C}ruz,
  {C}alif., 1979), 393--396, Proc. Sympos. Pure Math., Vol. 37, Amer. Math.
  Soc., Providence, RI, 1980.

\bibitem{BrouePuig}
M. Broué and L. Puig, \textit{A {F}robenius theorem for blocks}, Invent. Math.
  \textbf{56} (1980), 117--128.

\bibitem{Atlas}
J.~H. Conway, R.~T. Curtis, S.~P. Norton, R.~A. Parker and R.~A. Wilson,
  \textit{ATLAS of finite groups}, Oxford University Press, Eynsham, 1985.

\bibitem{Dade}
E.~C. Dade, \textit{Blocks with cyclic defect groups}, Ann. of Math. (2)
  \textbf{84} (1966), 20--48.

\bibitem{Dadeconj2}
E.~C. Dade, \textit{Counting characters in blocks with cyclic defect groups.
  {I}}, J. Algebra \textbf{186} (1996), 934--969.

\bibitem{Feit}
W. Feit, \textit{The representation theory of finite groups}, North-Holland
  Mathematical Library, Vol. 25, North-Holland Publishing Co., Amsterdam, 1982.

\bibitem{GAP47}
The GAP~Group, \textit{GAP -- Groups, Algorithms, and Programming, Version
  4.7.8}; 2015, (\url{http://www.gap-system.org}).

\bibitem{Gluck}
D. Gluck, \textit{Rational defect groups and 2-rational characters}, J. Group
  Theory \textbf{14} (2011), 401--412.

\bibitem{Gorenstein}
D. Gorenstein, \textit{Finite groups}, Harper \& Row Publishers, New York,
  1968.

\bibitem{HKS}
L. Héthelyi, B. Külshammer and B. Sambale, \textit{A
  note on Olsson's Conjecture}, J. Algebra \textbf{398} (2014), 364--385.

\bibitem{Kitaoka}
Y. Kitaoka, \textit{Arithmetic of quadratic forms}, Cambridge Tracts in
  Mathematics, Vol. 106, Cambridge University Press, Cambridge, 1993.

\bibitem{Kiyota}
M. Kiyota, \textit{On {$3$}-blocks with an elementary abelian defect group of
  order {$9$}}, J. Fac. Sci. Univ. Tokyo Sect. IA Math. \textbf{31} (1984),
  33--58.

\bibitem{KoshitaniSpath}
S. Koshitani and B. Späth, \textit{The inductive Alperin-McKay and blockwise
  Alperin weight conditions for blocks with cyclic defect groups},
  \href{http://arxiv.org/abs/1310.5512v1}{arXiv:1310.5512v1}.

\bibitem{Kuelshammer}
B. Külshammer, \textit{Crossed products and blocks with normal defect groups},
  Comm. Algebra \textbf{13} (1985), 147--168.

\bibitem{KuelshammerWada}
B. Külshammer and T. Wada, \textit{Some inequalities between invariants of
  blocks}, Arch. Math. (Basel) \textbf{79} (2002), 81--86.

\bibitem{Navarro}
G. Navarro, \textit{Characters and blocks of finite groups}, London
  Mathematical Society Lecture Note Series, Vol. 250, Cambridge University
  Press, Cambridge, 1998.

\bibitem{NavarrokB}
G. Navarro, \textit{Some remarks on global/local conjectures}, preprint.

\bibitem{NeukirchE}
J. Neukirch, \textit{Algebraic number theory}, Grundlehren der Mathematischen
  Wissenschaften, Vol. 322, Springer-Verlag, Berlin, 1999.

\bibitem{Plesken}
W. Plesken, \textit{Solving {$XX^\textnormal{tr}=A$} over the integers}, Linear
  Algebra Appl. \textbf{226/228} (1995), 331--344.

\bibitem{PuigMod}
L. Puig, \textit{Pointed groups and construction of modules}, J. Algebra
  \textbf{116} (1988), 7--129.

\bibitem{UsamiZ2Z2}
L. Puig and Y. Usami, \textit{Perfect isometries for blocks with abelian defect
  groups and {K}lein four inertial quotients}, J. Algebra \textbf{160} (1993),
  192--225.

\bibitem{Rouquiercyclic}
R. Rouquier, \textit{The derived category of blocks with cyclic defect groups},
  in: Derived equivalences for group rings, 199--220, Lecture Notes in Math.,
  Vol. 1685, Springer-Verlag, Berlin, 1998.

\bibitem{ExtraspecialExpp}
A. Ruiz and A. Viruel, \textit{The classification of {$p$}-local finite groups
  over the extraspecial group of order {$p^3$} and exponent {$p$}}, Math. Z.
  \textbf{248} (2004), 45--65.

\bibitem{habil}
B. Sambale, \textit{Blocks of finite groups and their invariants}, Springer
  Lecture Notes in Math., Vol. 2127, Springer-Verlag, Berlin, 2014.

\bibitem{SambaleC3}
B. Sambale, \textit{Cartan matrices and {B}rauer's {$k(B)$}-{C}onjecture
  {III}}, Manuscripta Math. \textbf{146} (2015), 505--518.

\bibitem{SambaleC4}
B. Sambale, \textit{Cartan matrices and {B}rauer's {$k(B)$}-{C}onjecture {IV}},
  J. Math. Soc. Japan (to appear).

\bibitem{Sambalerank3}
B. Sambale, \textit{On blocks with abelian defect groups of small rank},
  submitted.

\bibitem{Usami23I}
Y. Usami, \textit{On {$p$}-blocks with abelian defect groups and inertial index
  {$2$} or {$3$}. {I}}, J. Algebra \textbf{119} (1988), 123--146.

\bibitem{Watanabe1}
A. Watanabe, \textit{Notes on {$p$}-blocks of characters of finite groups}, J.
  Algebra \textbf{136} (1991), 109--116.

\bibitem{WatanabeSD16}
A. Watanabe, \textit{Appendix on blocks with elementary abelian defect group of
  order 9}, in: Representation {T}heory of {F}inite {G}roups and {A}lgebras,
  and {R}elated {T}opics ({K}yoto, 2008), 9--17, Kyoto University Research
  Institute for Mathematical Sciences, Kyoto, 2010.

\bibitem{WatanabeAWC}
A. Watanabe, \textit{The number of irreducible Brauer characters in a
  {$p$}-block of a finite group with cyclic hyperfocal subgroup}, J. Algebra
  \textbf{416} (2014), 167--183.

\end{thebibliography}
\end{document}